%% file: main.tex
\RequirePackage{fix-cm}
\documentclass[smallextended]{svjour3}       
\smartqed  
\usepackage{graphicx}
\usepackage{mathptmx}      
\usepackage{subcaption}
\usepackage{amsmath}
\usepackage{amssymb}
\usepackage{hyperref}
\makeatletter
\let\cl@chapter\undefined
\makeatletter
\usepackage{cleveref}
\usepackage{xcolor}
\definecolor{pec}{HTML}{FEC04C}

 \usepackage{amsthm}
 
\usepackage{thmtools, thm-restate}
\usepackage{tabularx}
\usepackage{todonotes}
\usepackage{optidef}
\usepackage{algpseudocode}
\usepackage{algorithm}
\usepackage{tikz}
\usetikzlibrary{hobby, shapes.geometric, arrows, arrows.meta, calc}
\usepackage{tikz-3dplot}
\usepackage{authblk}
\usepackage{makecell}

\usepackage[style=numeric,maxnames=10]{biblatex}
\usepackage{bm}
\newcommand{\real}{\mathbb{R}}
\newcommand{\complex}{\mathbb{C}}
\newcommand{\jim}{\bm{i}}
\newcommand\M{\mathcal{M}}%

\newcommand{\EFF}{E_{\mathrm{FF}}}
\newcommand{\Einc}{E^{\mathrm{inc}}}
\newcommand{\Eout}{E^\mathrm{out}}
\newcommand{\Bout}{B^\mathrm{out}}
\newtheorem{fact}{Fact}
\crefname{fact}{fact}{facts}
\Crefname{fact}{Fact}{Facts}

\renewcommand\Re{\operatorname{Re}}%
\renewcommand\Im{\operatorname{Im}}%
\newcommand\supp{\operatorname{supp}}%

\DeclareMathOperator{\rowsp}{rowsp}
\DeclareMathOperator{\Dom}{Dom}
\newcommand{\C}{\mathbb{C}}
\newcommand{\R}{\mathbb{R}}

\newcommand{\V}{\mathcal{V}}

\newcommand*\dd{\mathop{}\!\mathrm{d}}

\begin{document}
\tdplotsetmaincoords{60}{110}
\pgfmathsetmacro{\rvec}{.8}
\pgfmathsetmacro{\thetavec}{30}
\pgfmathsetmacro{\phivec}{60}

\title{Physics-Informed Mixed Integer Programming Methods for Optimal Patch Antennas
\thanks{This research is conducted as part of the Dutch National Growth Fund programme \mbox{POLARIS} (www.polaris-ngf.nl).
}
}

\titlerunning{Physics-Informed MIP Methods for Optimal Patch Antennas}        

\author{Ruben Hoeksma         \and
        Jesse van Rhijn       \and
        Marc Uetz             \and
        Wick Wijnholds
}


\institute{R.P. Hoeksma \at
              University of Twente, Faculty of Electrical Engineering, Mathematics and Computer Science \\
              \email{r.p.hoeksma@utwente.nl}           
           \and
           J. van Rhijn \at
              University of Twente, Faculty of Electrical Engineering, Mathematics and Computer Science\\
              and
              TNO, Radar Technology Technology Department, The Hague\\
              \email{j.vanrhijn@utwente.nl}
           \and
           M. Uetz \at
              University of Twente, Faculty of Electrical Engineering, Mathematics and Computer Science\\
              \email{m.uetz@utwente.nl}
           \and
           W.J. Wijnholds \at
              University of Twente, Faculty of Electrical Engineering, Mathematics and Computer Science\\
              \email{w.j.wijnholds@utwente.nl}
}

\date{
}

\maketitle

\begin{abstract}
Antennas come in various shapes and sizes, as their performance characteristics are very much dependent on their exact shape. 
Coming up with a shape that matches desired design characteristics has, up until this moment, been largely the domain of metaheuristics and local optimization methods. 
While these methods are convenient, they generally provide no performance guarantees and rarely allow to enforce hard design constraints.
In this paper, we show how Mixed-Integer Programming (MIP) can be used to optimize the shape of so-called patch antennas.
The methodological focus is to show how underlying laws of physics allow to derive effective bounds on variables as well as new cutting planes. These greatly improve the effectiveness of the underlying basic MIP model, enabling the automated design of practically relevant patch antennas for a wide variety of applications.
\end{abstract}

\section{Introduction}
\label{section: introduction}

Mixed-Integer Programming is one of the most important workhorses to solve complex
combinatorial optimization problems in operations research.
This success stems from increases in raw computational speed on the one hand, but more importantly from developments on the algorithmic side over the past decades.
Yet, in various domains of engineering, the optimization methods of choice remain heuristic algorithms.
In particular, antenna design is almost exclusively the domain of genetic optimization algorithms and their variants~\cite{lucchini2022topology}.

In part, this is due to the nature of the optimization problems involved, which are highly non-linear in both objective and constraints, and often cannot be easily formulated as mathematical optimization problems at all.
In~\cite{hoeksma2026solving}, however, we developed a nearly linear mixed-integer program (MIP) capable of optimizing simple electromagnetic wave scatterers such as a dipole antenna and the so-called Yagi-Uda antenna. This places a class of PDE-constrained design problems within the realm of applicability of mixed-integer programming.

The MIP developed in~\cite{hoeksma2026solving} was proved to be capable of optimizing simple, one-dimensional antenna geometries.
Most practically useful antennas are, however, \emph{two-dimensional} objects.
Of particular interest are \emph{patch antennas}, also known as \emph{microstrip antennas}~\cite{balanisAntennaTheoryAnalysis2012}.
These are antennas that can be printed directly on top of a printed circuit board (PCB), drastically reducing manufacturing cost and complexity.
One additional benefit of this approach is that these antennas can be printed in many different shapes, which allows the antenna engineer to tailor them for their application, since the performance characteristics of antennas are heavily dependent on their shape. 

In this paper, we extend the framework introduced in~\cite{hoeksma2026solving} to be able to solve two-dimensional antenna design problems.
This extension comes with several additional layers of complexity.
First, for planar structures, new electromagnetic constraints are needed to design practically useful antennas. 
Second, \emph{impedance matching} can be much more challenging in this setting.
Furthermore, the geometries require far more computational effort to solve their associated optimization problems to optimality.

To address these challenges, we introduce new convex constraints that allow us to include the electromagnetic design requirements of patch antennas.
Moreover, we derive a strengthened formulation than the one developed in~\cite{hoeksma2026solving}, by using the law of conservation of energy.
This enhanced formulation provides a far stronger dual bound compared to the initial formulation.
We view this as a form of \emph{physics-informed} mixed-integer programming, where the physical properties of antennas naturally provide alternative formulations with more favourable computational properties.

The rest of this paper is organized as follows:
First, with the aim of being  self-contained, in \Cref{section: background} we give a brief summary of the work in~\cite{hoeksma2026solving}.
Then, in \Cref{section: physicscuts}, we derive the energy conservation constraint.
In \Cref{section: constraints}, we introduce additional constraints which are derived from design requirements of antennas, including specifically new impedance and polarization constraints.
In \Cref{section: computation}, we validate the enhanced model by introducing the patch antenna test case. We demonstrate the improvements that can be obtained from the physics-derived constraints of \Cref{section: physicscuts} in terms of the strength of the formulation.
Finally, in \Cref{section: discussion}, we provide some concluding remarks.

\section{Background and Basic Model}
\label{section: background}

\subsection{Notation}
\label{section: notation}

Since summation indices are often denoted with the letter $i$, the imaginary unit is denoted by a boldface $\jim$.
We write the complex conjugate of a complex number $z \in \complex$ as $\bar z$,
and denote the Hermitian adjoint (that is, the complex conjugate transpose) of a matrix $A \in \complex^{n \times m}$ as $A^H$.\footnote{
    This notation is standard in the mathematics literature. However, physics and engineering works often denote the complex conjugate as $(\cdot)^*$ and the Hermitian adjoint as $(\cdot)^\dagger$.
}

We use the standard component-wise partial order for vectors; that is, if $x, y \in \real^n$, then $x \leq y$ means $x_i \leq y_i$ for all $i = 1, ..., n$.
For complex vectors we extend this in the natural way: if $x, y \in \complex^n$, then $x \leq y$ means $\Re\{x_i\} \leq \Re\{y_i\}$ and $\Im\{x_i\} \leq \Im\{y_i\}$ for all $i \in 1, ..., n$.

We will often work in spherical coordinates $(r, \vartheta, \varphi)$, where
\begin{align}
    \label{eq: spherical coordinates x}
    x &= r\sin\vartheta\cos\varphi, \\
    \label{eq: spherical coordinates y}
    y &= r\sin\vartheta\sin\varphi, \\
    \label{eq: spherical coordinates z}
    z &= r \cos\vartheta.
\end{align}
for some Cartesian coordinate system $(x, y, z)$ for $\R^3$.
Here, $r$ is the Euclidean distance from the origin, $\vartheta \in [0, \pi]$ is the polar angle with respect to the $z$-axis, and $\varphi \in [0, 2\pi)$ is the azimuthal angle with respect to the $x$-axis.

\subsection{Electromagnetism}
\label{section: electromagnetism}

The problem we tackle in this paper---the optimal design of patch antennas---is likely unfamiliar to the majority of mathematical programming specialists.
We must therefore first introduce several foundational concepts from electromagnetism and antenna engineering, before we can even state the optimization problem formally.
To keep this paper streamlined, the discussion in this section acts only as a high-level explanation of the required concepts.
For a more detailed discussion of the same concepts, we refer to~\cite{hoeksma2026solving} and the references therein.

\paragraph{Electric and magnetic fields.}
    In a nutshell, an antenna is a conducting (usually metallic) object to which a voltage oscillating at a natural frequency $\omega$ is applied.
    In response to this voltage, a current runs along the surface of the antenna, which in turn acts as a source for electric and magnetic fields.
    These fields are described by vector-valued functions, denoted by $\vec{E}$ and $\vec{B}$.
    The relationship between these fields and their sources is described by Maxwell's equations of electromagnetism.
    
    A fascinating consequence of these equations is that these fields can propagate as waves; we call this \emph{electromagnetic radiation}. 
    Not all fields generated by the antenna are in the form of electromagnetic waves; some of them merely oscillate close to the antenna. 
    \Cref{fig: blob} schematically depicts the phenomenon of electromagnetic \emph{wave scattering}, wherein an incoming wave interacts with a body to produce an outgoing (scattered) wave.

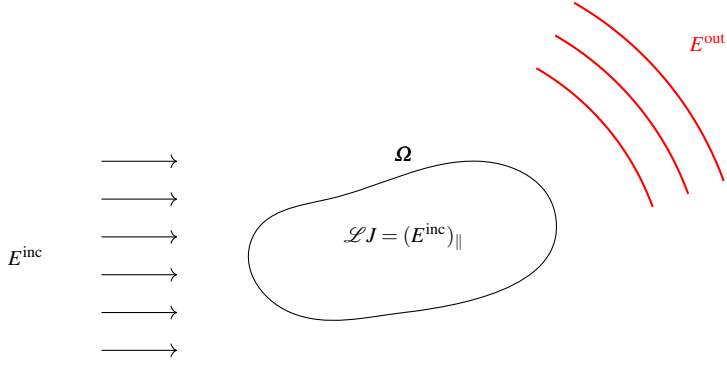
\begin{figure}
    \centering
    \begin{tikzpicture}
        \path[draw,use Hobby shortcut,closed=true]
            (0,0) .. (1,.5) .. (3,1) .. (4,.3) .. (2,-1) .. (.5,-1);
            \node[] (x) at (2, 1.1) {$\Omega$};

            \node at (2,1.1) {$\Omega$};
           
            \foreach \i in {-1, -0.5, 0, 0.5, 1, 1.5} {
                \draw[->] (-2, -\i) -- (-1, -\i);
            }

            \node at (2, 0) {$\mathcal{L}J = (\Einc)_\parallel$};
            
            \node[] (y) at (-3, -0.25) {$\Einc$};

            \draw [red,thick,domain=20:60] plot ({2+3.5*cos(\x)}, {-0.8+3.5*sin(\x)});
            \draw [red,thick,domain=20:60] plot ({2+4*cos(\x)}, {-0.8+4*sin(\x)});
            \draw [red,thick,domain=20:60] plot ({2+4.5*cos(\x)}, {-0.8+4.5*sin(\x)});
            
            \node[red] (Eout) at ({2+5.25*cos(40)}, {-0.8+5.25*sin(40)}) {$\Eout$};

    \end{tikzpicture}
    \caption{
        A general electromagnetic scattering problem.
        The incident field $\Einc$ induces a current density $J$ on the body $\Omega$, which radiates a field $\Eout$.
        The current density $J$ and the tangential component of the incoming field are related by a (linear) operator equation $\mathcal{L}J = (\Einc)_\parallel$.
        In the case of an antenna, $\Einc$ is provided by a time-harmonic voltage source.
    }
    \label{fig: blob}
\end{figure}
    
\paragraph{Far-field pattern.}
    In antenna design, we are interested in radiating fields.
    Therefore we measure the fields far away from the antenna, where only the radiation is measurable.
    We call this the \emph{far-field region}.
    In the far-field region, the magnitudes of the electric and magnetic fields are related by a simple constant (see \Cref{fact: far-field}) so we only need to evaluate the electric field. 
    
    Antenna engineers are particularly interested in the way this field varies over different directions.
    Imagine a sphere centred on the antenna, with a radius much larger than the antenna.
    If we evaluate the electric field only on this spherical shell, the result is the \emph{far-field pattern}, defined as
    \[
        \EFF(\vartheta, \varphi) = \lim_{\|r\| \to \infty} \|r\|\vec{E}(\vec r),
    \]
    with the position vector $\vec r = (x, y, z)$ written in spherical coordinates, cf.~\Cref{eq: spherical coordinates x,eq: spherical coordinates y,eq: spherical coordinates z}.
    Here, $(\vartheta, \varphi)$ are the polar and azimuthal angle of $\vec{r}$ in a suitably chosen spherical coordinate system. 
    This function captures exactly the properties of the radiated fields relevant to antenna design.
    
\paragraph{Voltage, current, and complex numbers.}
	Antennas are made from a conductive material (usually, a metal like copper or aluminum).
    When an oscillating voltage $V$ is applied to the antenna at some angular frequency $\omega > 0$,
    a current $I$ is induced that oscillates at the same frequency, but possibly with a different phase~\cite{balanisAntennaTheoryAnalysis2012}.
    It is therefore common practice to represent both as complex numbers: the amplitude then represents the amplitude of the oscillation, while the argument represents the phase shift relative to some reference.
    Consequently, also the oscillating electric and magnetic fields are written as complex-valued vector functions, encoding both the amplitude and phase information of the radiating waves.

\paragraph{Input impedance and reflection.}
    The voltage applied to the input (or \emph{feed}) of the antenna is known as the feed voltage $V_{\mathrm{feed}}$; the current induced by this voltage is the feed current $I_{\mathrm{feed}}$.
    The ratio of these is called the \emph{input impedance} of the antenna,
	\[
		Z_{\mathrm{in}} = \frac{V_{\mathrm{feed}}}{I_{\mathrm{feed}}}.
	\]
	Note that $Z_{\mathrm{in}}$ is a complex number; it can be seen as a time-harmonic analogue to electrical resistance~\cite{balanisAntennaTheoryAnalysis2012}.

    The input impedance is an especially important quantity in antenna design, since also the \emph{feed line} to which the antenna is connected has a characteristic impedance, denoted $Z_0$.
    When these two impedances are nearly equal, power can flow unimpeded from the feed line to the antenna.
    Conversely, when the impedances are mismatched, only a small fraction of the supplied power is accepted by the antenna, while the remainder is \emph{reflected}.
    Hence, ensuring that the impedance of an antenna is matched to that of its feed line is essential; this is known as \emph{impedance matching}~\cite{balanisAntennaTheoryAnalysis2012}.

\paragraph{Meshes.}
    Maxwell's equations of electromagnetics are a system of coupled partial differential equations (PDEs).
    Closed-form solutions for these equations are only known in highly symmetric situations, or under restrictive assumptions. 
    Therefore, solutions to these equations are estimated by using computational methods, which require discretization of the antenna geometry.
    Typically, this is done by approximating the antenna as a \emph{mesh} $\M = (\V, T)$, where $\V$ is a set of vertices and $T \subseteq \binom{\V}{3}$ is a set of triangles.
    
    For our purposes, it is convenient to be able to speak of a \emph{submesh} of a mesh $\M$:
    If $\M' = (\V,  T')$ is such that $\V' \subseteq \V$ and $T' \subseteq T$,  then we say that $\M'$ is a submesh of $\M$, and write $\M' \subseteq \M$.

\paragraph{Boundary Element Method.}
    Given a mesh $\M$, Maxwell's equations can be solved approximately using the \emph{Boundary Element Method} (BEM).
    In this method, one associates with $\M$ a basis of functions $B_\M$ and a matrix $Z_\M \in \complex^{B_\M \times B_\M}$.
    The \emph{excitation} incident on the antenna, for example due to a voltage source or an incident signal, yields a vector $V \in \complex^{B_\M}$.
    Then, the current density $J$ on the surface of the antenna induced by the incident electric field can be approximated as
    \[
        J = \sum_{f \in B_\M} I_f f,
    \]
    where the coefficients $I_f \in \complex$ are found as the (unique) solution to the linear system
	\[
		Z_{\M}I_{\M} = V_\M.
	\]
    For a more detailed discussion on how to choose $B_\M$ and use it to compute $Z_\M$ and $V_\M$ is computed, we once again refer to~\cite{hoeksma2026solving}.
    In the following, one can simply assume that given $\M$, there is a systematic procedure to obtain these data.

    The current coefficients $I \in \complex^{B_\M}$ also allow one to easily compute the far-field pattern produced by the associated current density as a linear form.
    Let 
    \[
    \EFF(\vartheta, \varphi)~:~\mathbb{S}^2~\to~\complex^3
    \] 
    be the far-field pattern due to $J$, and let $e_\alpha \in \real^3$ be any unit vector.
    Then the component of $\EFF(\vartheta, \varphi)$ parallel to $e_\alpha$ is
    \[
        e_\alpha \cdot \EFF(\vartheta, \varphi) = c_\M(\vartheta, \varphi)^T I.
    \]
    The (continuous) functions $c_\M(\vartheta, \varphi): \mathbb{S}^2 \to \complex^{B_\M}$ depend only on the mesh and its associated basis $B_\M$, and are straightforward to compute.

\subsection{Sub-Mesh Selection Problem}

From the preceding section, we can understand that different meshes $\M$ may induce different far-field patterns $\EFF$ for the same physical excitation.
As the far-field pattern is of particular interest to antenna designers, one may ask what antenna shape, subject to reasonable constraints, produces the best far-field pattern for a given purpose.

Here, as in~\cite{hoeksma2026solving}, we consider a simple yet flexible version of this question: Given a fixed mesh $\M$, which submesh $\M' \subseteq \M$ maximizes the magnitude of a specified component of the far-field pattern?
This leads to the \emph{sub-mesh selection problem}:
\begin{quote}
	For a given mesh $\M$, applied voltage $V$, unit vector $e_\alpha \in \real^3$ and direction $\hat{r}_{\vartheta, \varphi} \in \mathbb{S}^2$, find a submesh $\M' \subseteq \M$ that maximizes the $e_\alpha$-component of electric field in the far-field region in the direction $\hat{r}_{\vartheta, \varphi}$.
\end{quote}
Recall that given $\M$ and $\hat{r}_{\vartheta, \varphi}$, we can compute the far-field pattern through a well-defined procedure: Fix a basis $B_\M$, compute the matrix $Z_\M$ and vectors $V_\M$ and $c_\M(\vartheta, \varphi)$, and solve the linear system $Z_\M I = V_\M$; then $e_\alpha \cdot \EFF(\vartheta, \varphi) = c_\M^T I$.
Therefore, the sub-mesh selection problem can be written formally as
\begin{maxi*}|s|
	{\scriptstyle \M' \subseteq \M}{|c_{\M'}^T I|}{}{}
	\label{eq: natural formulation}
	\tag{N}
	\addConstraint{Z_{\M'}I}{=V_{\M'}}{}
	\addConstraint{I}{\in \C^{B_{\M'}}.}{}
\end{maxi*}

Note that this optimization problem is not in the form of a standard mathematical program, since the constraint matrix and the dimensionality of the vector $I$ depend on the feasible solution $\M$.
Moreover, the objective function, being the absolute value of a complex number, is non-linear in $I$, and its dependence on $\M$ is not immediately clear.
Thus, it is not a priori obvious how to approach solving this problem.

\subsection{The Foundational Theorem}

In~\cite{hoeksma2026solving}, we exhibit a reformulation of Problem \eqref{eq: natural formulation} as an MISOCP.
This requires three crucial observations.
First, we find that the simulation data $Z_{\M'}$ and $I_{\M'}$ of a submesh $\M'\subseteq \M$ can be obtained from $Z_\M$ and $I_\M$ by removing the rows and columns that correspond to $T \setminus T'$.
Second, if any triangle in the support of a function $f \in B_\M$ is not present in a submesh $\M'$, then the vector $I$ corresponding to $\M'$ has $I_f = 0$.\footnote{
    This criterion has a simple physical interpretation: If a triangle is not present in $\M'$, then $\M'$ cannot support any current on this triangle.
}
Third, we replace the non-linear form $|c_\M^T I|$ by $\Re\{c_\M^T I\}$, by ensuring that $c_\M^T I$ is purely real in any optimal solution.
These observations lead to the following theorem.

\begin{theorem}[\cite{hoeksma2026solving}]\label{thm: base mip}
	Fix a mesh $\M = (\V, T)$, and let $Z = Z_\M$ and $V = V_\M$.
	The optimal objective value of the antenna design problem  $\max \{|c_{\M'}^T I| \, : \, \M'\subseteq \M, \, Z_{\M'}I = V_{\M'}, \, I \in \complex^{B_{\M'}}\}$ is equal to the optimal objective value of
	\begin{maxi*}|s|
		{x \in \{0, 1\}^{T(\M)}}{\Re\{c_\M^T I\}}{}{}
        \tag{M}
        \label{eq: base mip}
		\addConstraint{y_f = 1}{\implies Z_{f, *}I = V_f \cdot z}{\quad \forall f \in B_\M}
		\addConstraint{y_f = 0}{\implies I_f = 0}{\quad \forall f \in B_\M}
		\addConstraint{\prod_{t \in \supp(f)} x_t}{= y_f}{\quad \forall f \in B_\M}
		\addConstraint{|z|}{\leq 1}{\quad z \in \complex.}
	\end{maxi*}
\end{theorem}

The following is easily shown from the form of Program \eqref{eq: base mip}.

\begin{corollary}
    \label{cor: objective real}
    In any optimal solution $(x^*, I^*)$ to Program \eqref{eq: base mip}, it holds that $c_\M^T I^*$ is real and non-negative.
\end{corollary}

Program \eqref{eq: base mip} of \Cref{thm: base mip} is an MISOCP with only a single convex quadratic constraint; such programs can be solved effectively by many standard MIP solvers.
Unfortunately, the model contains many indicator constraints.
Solvers typically  reformulate indicator constraints as big-$M$ constraints, but this relies on the solver being able to infer sufficiently strong bounds on the variables involved in the indicator constraints~\cite{gurobi,williams1990model}.
This is generally not the case for Program \eqref{eq: base mip}, so without strengthening the program further it is only suitable for very simple antenna design problems.

\section{Physics-Informed Cuts}
\label{section: physicscuts}

Some of the most powerful laws of physics are the \emph{conservation laws}, which specify that a certain quantity remains constant in over time. 
When analyzing physical systems, conservation laws are often a very useful tool, allowing one to deduce properties of the system from these laws alone.
Since Program \eqref{eq: base mip} heuristically has the laws of electromagnetism as its constraints, one may ask whether explicitly adding more physical laws, in particular conservation laws, provides stronger formulations.
Showing that this is indeed the case is the main goal of this paper.

More specifically, we apply the law of conservation of energy for electromagnetic systems, and from it derive variable bounds and additional convex constraints, which significantly strengthen the MISOCP of \Cref{thm: base mip}.
These bounds and constraints each depend on an upper bound $P_\mathrm{max}$ on the energy that enters the antenna per unit time.
A finite value for $P_\mathrm{max}$ can be found from the impedance constraints that are derived in \Cref{sec: impedance}.
All of this follows from the laws of physics; we therefore regard these constraints as physics-informed cuts.

\subsection{Energy Conservation}
\label{sec: energy conservation}

For the convenience of the reader, we now state the theory of energy conservation for electromagnetic systems.
The result from electromagnetics that we need is a version of \textit{Poynting's Theorem}~\cite[eq. 1.9.6]{orfanidisElectromagneticWavesAntennas2016}, which relates the flow of electromagnetic energy out of a volume $V$ through its boundary $\partial V$ to the electric fields and their generated currents within this volume:
\begin{equation}
    \frac{1}{\mu_0}\int_{\partial V} \Re\left\{E \times \overline{B}\right\}\cdot n dA = -\int_V \Re\left\{J \cdot \overline{E}\right\} dx, \label{equation: poynting}
\end{equation}
Here $n$ is an outward normal vector for $\Sigma$, the surface of the antenna. Recall that $E, B$ respectively are the electric and magnetic field and $J$ is the current density.
We now massage this version of Poynting's Theorem into a form that makes it practically usable for inclusion in the MIP from \Cref{thm: base mip}. Specifically, we get rid of the magnetic field $B$ and obtain an expression only in $\EFF, J$ and $\Einc$ (the incident electric field), and then rewrite this into an expression in the MIP variables.

Evaluating the left integral of \eqref{equation: poynting} for $\Eout$ (the electric field produced by the antenna) on a sphere of radius $R$ containing $\Sigma$ yields
\begin{multline}
    \frac{1}{\mu_0}\int_{-\pi}^\pi\int_0^\pi R^2\Re\left\{\Eout(R\hat{r}_{\vartheta, \varphi}) \times \overline{\Bout(R\hat{r}_{\vartheta, \varphi})}\right\}\cdot \hat{r}_{\vartheta, \varphi} \sin \vartheta ~\dd\varphi \dd\vartheta \\= -\int_{\Sigma} \Re\left\{J \cdot \overline{\Eout}\right\} \dd x, 
\end{multline}
where $B^{out}$ is the magnetic field generated by $J$, and we use that, for our antenna problem, currents only run on $\Sigma$.
We also need the following fact:
\begin{fact}
\label{fact: no tangential electric field}
    The electric field on the boundary of a perfect electrical conductor has no tangential component~\cite{balanisAntennaTheoryAnalysis2012,davidsonComputationalElectromagneticsRF2010}, which implies that, on $\Sigma$:
    \[(\Eout)_{||} +(\Einc)_{||} = 0,\]
    where $(\Eout)_{||}$ is the component of $\Eout$ tangential to $\Sigma$.
\end{fact}

Using \Cref{fact: no tangential electric field}, we can replace the right hand side: 
\begin{multline}
    \frac{1}{\mu_0}\int_{-\pi}^\pi\int_0^\pi R^2\Re\left\{\Eout(R\hat{r}_{\vartheta, \varphi}) \times \overline{B^{out}(R\hat{r}_{\vartheta, \varphi})}\right\}\cdot \hat{r}_{\vartheta, \varphi} \sin \vartheta ~\dd\varphi \dd\vartheta \\= \int_{\Sigma} \Re\left\{J \cdot \overline{\Einc}\right\} \dd x.
\end{multline}
Now we use the following well-known property of the far-field pattern~\cite[sec. 15.9]{orfanidisElectromagneticWavesAntennas2016}:
\begin{fact}
\label{fact: far-field}
    In the far-field region, the magnetic and electric field are each tangential to a sphere centered at their source. Furthermore, they are orthogonal and proportional by a factor $\eta_0 = \sqrt{\frac{\mu_0}{\epsilon_0}}$, so that
    \begin{align}
        &\lim_{R\rightarrow\infty} R^2\Re\left\{\Eout(R\hat{r}_{\vartheta, \varphi}) \times \overline{B^{out}(R\hat{r}_{\vartheta, \varphi})}\right\} \\
        =~&\lim_{R\rightarrow\infty} \frac{1}{\eta_0}\left(R^2 \|\Eout(R\hat{r}_{\vartheta, \varphi})\|^2\right) = \frac{1}{\eta_0}\|\EFF(\vartheta, \varphi)\|^2.
    \end{align}
\end{fact}
Together, this yields the following equation for energy conservation:
\begin{equation}
\label{eq: poynting far field}
    c_0\int_{-\pi}^\pi\int_0^\pi \|\EFF(\vartheta, \varphi)\|^2 \sin \vartheta ~\dd\varphi \dd\vartheta = \int_{\Sigma} \Re\left\{J \cdot \overline{\Einc}\right\} \dd x,
\end{equation}
where $c_0$ is the speed of light in a vacuum, defined as 
\[
    c_0 = \frac{1}{\sqrt{\mu_0\epsilon_0}} = \frac{1}{\mu_0\eta_0}.
\]
Using \Cref{eq: poynting far field} we can derive a valid inequality for the program of \Cref{thm: base mip}.

\begin{theorem}
\label{lem: energy conservation}
    Let $(x, I, z) \subseteq \{0, 1\}^{B_\M} \times \complex^{B_\M} \times \complex$ be feasible for Program \eqref{eq: base mip}.
    Then the radiated energy equals the input power, that is, 
    \[
        I^H R I = \Re\{(Vz)^H I\}.
    \]
    Here, the (self-adjoint) radiation matrix $R \in \complex^{B_\M \times B_\M}$ with entries indexed by basis functions
    \[
        R_{f,g} = c_0\int_{\mathbb{S}^2} \sum_{\alpha \in \{x, y, z\}}\overline{c^\alpha_f(r)} \cdot c^\alpha_g (r) \dd{r}
    \]
    is positive semi-definite.
    Consequently, if the input power $\Re\{(Vz)^H I\}$ is bounded from above by $P_\mathrm{max}$, then $I^H R I \leq P_\mathrm{max}$ is a valid convex inequality for Program \eqref{eq: base mip}.
\end{theorem}

\begin{proof}
    Let $z \in \complex$ with $|z| \leq 1$.
    First, observe that if $z = 0$, then $\tilde{I} = 0$, and so the identity from the lemma holds trivially.
    Assume then henceforth that $z \neq 0$. 
    
    In~\cite{hoeksma2026solving} it is shown that for any such $z$, the following function $\phi_z$ transforms a solution $(\M', I)$ to the antenna optimization problem \eqref{eq: natural formulation} to a solution $(x, \tilde I, z)$ to Program \eqref{eq: base mip}.
    A mesh $\M' \subseteq \M$ is mapped to a binary vector $x \in \{0, 1\}^T$ by setting $x_t = 1$ if and only if $t \in T(\M')$; that is, $x$ is the incidence vector of the sub-mesh $\M'$ of $\M$.
    A vector $I \in \complex^{\M'}$ is mapped to $\tilde I \in \complex^{B_\M}$ by setting 
    \[
        \tilde I_f = \begin{cases} 
            I_f \cdot z & \text{if $f \in B_{\M'}$}, \\
            0 & \text{otherwise}.
        \end{cases}
    \]
    For a solution $(x, \tilde I ,z)$ to Program \eqref{eq: base mip}, let $(\M', I) = \phi_z^{-1}(x, \tilde I)$ be given by
    \[
        \M' = (\V, \{t \in T(\M) \, : \, x_t = 1\})
    \]
    and $I_f = \tilde I_f/z$ for each $f \in B_{\M'}$.
    Then $\phi_z$ allows us to derive the inequality from the theorem for an arbitrary sub-mesh $\M'$ of $\M$.

    Let $\M' \subseteq \M$ and let $(\M', I)$ be feasible for Program \eqref{eq: natural formulation}, so that $Z_{\M'} I = V_{\M'}$.
    Let $J = \sum_{f \in B_{\M'}} I_f f$.
    From \Cref{eq: poynting far field},
    \[
        c_0\int_{\mathbb{S}^2} \overline{E_{\text{FF}}}(x) \cdot E_{\text{FF}}(x) \dd{x} = \int_{\Sigma}\Re\left\{J \cdot \overline{E^\mathrm{inc}}\right\}\dd{x},
    \]
    with $E_{\text{FF}}$ the far-field electric field as defined in \Cref{section: electromagnetism}.
    Recall also from \Cref{section: electromagnetism} that the $\alpha$-component of the field generated by the current density $J$ is given by $e_\alpha \cdot E_\text{FF} = \sum_{f \in B_{\M'}} c^\alpha_f I_f$, and so
    \begin{align*}
        &c_0\int_{\mathbb{S}^2} \overline{E_{\text{FF}}}(x) \cdot E_{\text{FF}}(x) \dd{x}
            \\ =~&c_0\sum_{f \in B_{\M'}} \sum_{g \in B_{\M'}}\overline{I_f} I_g
            \int_{\mathbb{S}^2} \sum_{\alpha \in \{x, y, z\}}  \overline{c^\alpha_f(x)} \cdot c^\alpha_g(x) \dd{x} = I^H R^{\M'} I.
    \end{align*}
    where $R^{\M'}_{fg} = \int_{\mathbb{S}^2} \overline{c_f(x)} \cdot c_g(x) \dd{x}$ for $f, g \in B_{\M'}$.
    
    Observe that $R^{\M'}$ is positive semi-definite for any $\M'$, since $\int_{\mathbb{S}^2} \overline{E_\text{FF}}(x) \cdot E_\text{FF}(x) \dd{x}$ is non-negative for any set of expansion coefficients $I$.
    Inserting $J = \sum_{f \in B_{\M'}} I_f f$ into $\int_\Sigma \Re\left\{J\cdot \overline{E^\mathrm{inc}}\right\}\dd{x}$ yields
    \[
        \int_\Sigma J(x) \cdot \overline{E}^\text{inc}(x) \dd{x} = \sum_{f \in B_{\M'}} I_f \int_\Sigma
            f(x) \cdot \overline{E}^\text{inc}(x) \dd{x} = V_{\M'}^H I,
    \]
    since the vector $V_{\M'}$ has entries $V_{\M', f} = -\int_\Sigma f(x) \cdot E^\text{inc}(x) \dd{x}$ for each $f \in B_{\M'}$ (cf.~\cite [Section 2.3]{hoeksma2026solving}).
    Putting the two parts together, we find
    \[
        I^H R^{\M'} I = \Re\{V_{\M'}^H I\}.
    \]

    Now, we only need to translate this identity to $(x, \tilde{I}, z) = \phi_z(\M', I)$.
    From the construction of $\phi_z$, for any $f \in B_{\M'}$, we have $\tilde{I}_f = I_f\cdot z$, while for $f \notin B_{\M'}$  we have $I_f = 0$.
    Thus, letting $R \in \complex^{B_\M}$ be the matrix from the lemma statement, we have
    \begin{align*}
        I^H R^{\M'} I  &= \sum_{f \in B_{\M'}}\sum_{g \in B_{\M'}} \overline{I_f} R^{\M'}_{f, g} I_g
            = 
            \frac{1}{z\overline{z}} \cdot
            \sum_{f \in B_\M}\sum_{g \in B_\M} \overline{\tilde{I}_f} R_{f, g} \tilde{I}_g \\
            &=  
            \Re\left\{\sum_{f \in B_{{\M'}}} \overline{V_{\M', f}} I_f\right\} = \Re\{V^H \tilde{I} / z \}.
    \end{align*}
    Therefore,
    \[
        \tilde{I}^H R \tilde{I} = z\overline{z} \cdot \Re\left\{
            V^H \tilde{I}/z 
        \right\}
        =
        \Re\left\{V^H \tilde{I} \cdot \frac{z\overline{z}}{z}\right\}
            = \Re\left\{
                (Vz)^H \tilde{I}
            \right\},
    \]
    completing the proof. 
\end{proof}

While the inequality $I^H R I \leq P_\mathrm{max}$ holds for any $I$ derived from an integer feasible solution to Program \eqref{eq: base mip}, it may be violated for fractional solutions.
This occurs because $I$ may longer satisfy $Z_{\M'} I = V_{\M'}$ for any submesh $\M'$, and thereby the link between the antenna geometry $\M$ and the current density $J = \sum_{f \in B_\M} I_f f$ is lost.

Consider then an optimal fractional solutions with $I^H R I > P_\mathrm{max}$.
Since \Cref{lem: energy conservation} is a discrete form of the law of energy conservation, such a solution can be interpreted as an unphysical current density that radiates more energy than is supplied to the antenna.
By adding the convex inequality $I^H R I \leq P_\mathrm{max}$ to Program \eqref{eq: base mip}, solvers can exclude such unphysical fractional solutions a priori.

\subsection{Current Decomposition}
\label{sec: decomposition}

While the energy conservation cut derived in \Cref{sec: energy conservation} is rather useful on its own, it also motivates us to consider which components of the current coefficients $I$ contribute significantly to the objective function and to the radiated energy.
As discussed in \Cref{section: electromagnetism}, the current density $J$ on the surface of an antenna radiates electromagnetic energy; however, not all current densities radiate~\cite{schab2018energy}.

In terms of the current coefficients $I \in \complex^{B_\M}$, we can decompose these into two terms,
\[
    I = I_\parallel + I_\perp,
\]
with the property that $|c_\M(\vartheta, \varphi)^T I| = |c_\M(\vartheta, \varphi)^I I_\parallel|$ for any $(\vartheta, \varphi)$.
Under some circumstances, this yields a non-trivial decomposition; the key is the form of the matrix $R$ from \Cref{lem: energy conservation}.

\begin{lemma}
\label{lem: ker R}
    Let $R \in \complex^{B_\M \times B_\M}$ be the self-adjoint matrix of \Cref{lem: energy conservation}, and assume that for each $\alpha \in \{x, y, z\}$, the function $c^\alpha_\M : \mathbb{S}^2 \to \complex^{B_\M}$ is continuous.
    Then $I_\perp \in \ker R$ if and only if $c^\alpha_\M(\vartheta, \varphi)^T I_\perp = 0$ for all $(\vartheta, \varphi)$.
\end{lemma}

\begin{proof}
    First, assume that for each $\alpha \in \{x, y, z\}$ we have $c^\alpha_\M(\vartheta, \varphi)^T I_\perp = 0$ for all $(\vartheta, \varphi)$.
    Then
    \begin{align*}
        (RI_\perp)_f &= \sum_{g \in B_\M} R_{f,g} I_{\perp, g} = \sum_{g \in B_\M} \left(\int_{\mathbb{S}^2} \sum_{\alpha \in \{x, y, z\}} \overline{c^\alpha_f(x)} \cdot c^\alpha_g (x) \dd{x}\right) I_{\perp, g} \\
        &= \int_{\mathbb{S}^2} \sum_{\alpha \in \{x, y, z\}} \overline{c^\alpha_f(x)} \cdot \left(\sum_{g \in B_\M} c^\alpha_g(x)I_{\perp, g}\right) \dd{x}
        = 0,
    \end{align*}
    since $\sum_{g \in B_\M} c^\alpha_g(x) I_{\perp, g} = c^\alpha_\M(\vartheta,\varphi)^T I_\perp = 0$.
    This proves one direction.

    Now for the other direction, assume that $I_\perp \in \ker R$.
    Then $I^H_\perp R I_\perp = 0$.
    From the proof of \Cref{lem: energy conservation}, we know that
    \[
        I^H_\perp R I_\perp = \int_{\mathbb{S}^2}  \|\EFF^\perp(x)\|^2 \dd{x},
    \]
    where $\EFF^\perp$ is the far-field electric field generated by the current density $J_\perp = \sum_{f \in B_\M} I_{\perp, f} f$.
    We also know that
    \[
        e_\alpha \cdot \EFF^\perp(\vartheta, \varphi) = c^\alpha_\M(\vartheta, \varphi)^T I_\perp.
    \]
    Therefore, the integral above can only vanish if for each $\alpha \in \{x, y, z\}$, it holds that $c^\alpha_\M(\vartheta, \varphi)^T I$ vanishes everywhere except for possibly a set of measure zero.
    Since each $c^\alpha_\M(\vartheta,\varphi)$ is a continuous function, $\|\EFF^\perp(\vartheta, \varphi)\|^2\sin\vartheta$ is also a continuous function.
    This implies that $\EFF^\perp (\vartheta,\varphi)\sin\theta$ vanishes everywhere, and therefore that $\EFF^\perp$ vanishes everywhere.
\end{proof}

\Cref{lem: ker R} shows that, whenever $\dim \ker R > 0$, there exist non-trivial current densities that do not contribute to the objective function $\Re\{c_\M^T I\}$; and, on the other hand, all such current densities can be recovered from $\ker R$.
This suggests an alternative formulation of Program \eqref{eq: base mip}.
The below formulation is not the most compact: one can easily eliminate $I$ by using the appropriate coordinate transformation matrices, thus obtaining an equivalent representation with $2|B_\M|$ fewer variables.

\begin{theorem}
\label{thm: decomposed formulation}
    Let $R$ be as in \Cref{lem: energy conservation}.
    Let $B_\parallel$ be a basis for $\rowsp R$ and $B_\perp$ a basis for $\ker R$.
    Then the program of \Cref{thm: base mip} is equivalent to
    \begin{maxi*}|s|
        {\scriptstyle x \in \{0, 1\}^{T(\M)}}{\Re\{c_\M^T I_\parallel\}}{}{}
            \label{eq: decomposed formulation}
            \tag{D} 
        \addConstraint{y_f = 1}{\implies Z_{f, *}I = V_f \cdot z}{\quad \forall f \in B_\M}
        \addConstraint{y_f = 0}{\implies I_f = 0}{\quad \forall f \in B_\M}
        \addConstraint{\prod_{t \in \supp(f)} x_t}{= y_f}{\quad\forall f \in B_\M}
        \addConstraint{|z|}{\leq 1}{\quad z \in \complex}
        \addConstraint{I}{= I_\parallel + I_\perp}{}
        \addConstraint{u^T I_\parallel}{= 0}{\quad \forall u \in B_\perp}
        \addConstraint{u^T I_\perp}{=0}{\quad\forall u \in B_\parallel.}
    \end{maxi*}
    
\end{theorem}

This theorem follows directly from \Cref{lem: ker R}.
While this reformulation has the same strength as the one from \Cref{thm: base mip}, it does allow us to obtain some other useful properties of the antenna design formulation.
Chief among these is \emph{variable bounds}.

In the decomposed program \Cref{eq: decomposed formulation} only $I_\parallel$ enters the objective function.
Deriving valid bounds for $I_\parallel$ thus ensures a valid dual bound.
By restricting the matrix $R$ to its row space, such bounds can be readily obtained, under the assumption of bounded radiated power $P_\mathrm{max}$.
This assumption is justified in \Cref{sec: impedance}.

\begin{lemma}
\label{lem: variable bounds}
    Let $R$ be as in \Cref{lem: energy conservation}, and let $I_\parallel \in \rowsp R$.
    Let $F \in \complex^{B_\M \times \rowsp R}$ be a matrix whose columns form an orthonormal basis for the row space of $R$, and let $R_\parallel = F^T R F$.
    If $P_\mathrm{max}$ is an upper bound to the input power of the antenna, 
    then for any vector $v \in \complex^{\rowsp R}$,
    \[
        |v^H I_\parallel| \leq \sqrt{v^H  R_\parallel^{-1} v\cdot P_\mathrm{max}}. 
    \]
    In particular, on choosing an orthonormal basis $B_\parallel$ for $\rowsp R$, for any $e_f \in B_\parallel$ we have
    \[
        |e_f^T I_\parallel| = |I_{\parallel, f}| \leq \sqrt{R_{\parallel, ff}^{-1}P_\mathrm{max}}.
    \]
\end{lemma}

\begin{proof}
    First, note that $R_\parallel$ is a linear map with $\Dom R_\parallel = \Im R_\parallel = \rowsp R$.
    Therefore, for any $I_\parallel \in \rowsp R \setminus \{0\}$, we have
    \[
        I_\parallel^H R_\parallel I_\parallel > 0;
    \]
    that is, $R_\parallel$ is positive definite.
    Moreover, observe:
    \[
        I_\parallel^H R_\parallel I_\parallel = I_\parallel^H F^T R F I_\parallel
         = (FI_\parallel)^H R (FI_\parallel) = I^H R I \leq P_\mathrm{max},
    \]
    for any $I \in \complex^{B_\M}$ such that $I = FI_\parallel + I_\perp$, with $I_\perp \in \ker R$.
    Here, the last inequality follows from \Cref{lem: energy conservation}.

    Now we consider $|v^H I_\parallel|$.
    Applying the Cauchy--Schwarz inequality, we find
    \begin{align*}
        |v^H I_\parallel|^2 &= |v^H R_\parallel ^{-1/2} R^{1/2}_\parallel I_\parallel|^2
            \leq \|v^H R_\parallel^{-1/2}\|^2 \cdot \|R_\parallel^{1/2} I_\parallel\|^2\\
            &= v^H (R^{-1/2})^H R^{-1/2} v \cdot I_\parallel^H R_\parallel I_\parallel 
            \leq |v^H R_\parallel^{-1} v| P_\mathrm{max},
    \end{align*}
    where we use the fact that any positive-definite matrix has an invertible square root and the fact that $R_\parallel$ is self-adjoint.
\end{proof}

This lemma also yields as a corollary an immediate dual bound on the objective value of \eqref{eq: base mip}.

\begin{corollary}
\label{cor: dual bound}
    Let $F$ and $P_\mathrm{max}$ be as in \Cref{lem: variable bounds} and let $c_\M^\parallel = Fc_\M$.
    Then the objective value of \eqref{eq: base mip} is bounded from above by
    \[
        \sqrt{|(c^\parallel_\M)^T R_\parallel^{-1} \overline{(c^\parallel_\M)}|P_\mathrm{max}}.
    \]
\end{corollary}

Although this dual bound is implied already by \Cref{lem: energy conservation}, some solvers may not be able to prove the bound from their specific implementation of convex quadratic constraints.
Thus, adding the (rather cheap) cut of \Cref{cor: dual bound} can be beneficial.

\begin{remark}
    The formulation of \Cref{thm: decomposed formulation} can be especially useful in cases where $\ker R$ is nontrivial.
    In the case of an antenna described by an open surface with a surface-supported current density of finite extent, $\ker R$ is always trivial; this follows from the injectivity of the electric field integral operator~\cite[Theorem 3.2]{buffa2003electric} and Rellich's Lemma~\cite[Lemma 2.9]{colton2015inverse}.
    
    Two important cases where $\ker R$ is nontrivial are volumetric current densities~\cite{devaney1973radiating} and infinite antenna arrays~\cite{schab2022sparsity,schab2023scattering}, which can both be modelled using the same framework as described in \Cref{section: electromagnetism} and in~\cite{hoeksma2026solving}.
    Even when $\ker R$ is trivial, $R$ can be numerically singular, in which case \Cref{thm: decomposed formulation} can still be useful to extract finite dual bounds via \Cref{lem: variable bounds,cor: dual bound}.
\end{remark}

\subsection{Physically-Motivated Variable Bounds}
\label{sec: heatingbounds}

\Cref{lem: variable bounds} provides exact upper bounds on the current density coefficients.
However, these bounds are rather loose.
For the patch antenna that we consider in \Cref{section: computation}, the largest values reach up to $10^{11}$ A/m when the input voltage is $1 \, \mathrm{V}$.
These bounds are so large that most solvers are not able to use them to effectively reformulate the indicator constraints of Program \eqref{eq: base mip}.

There are no fundamental physical limits on the current density at a single point on a perfectly electrically conducting antenna~\cite{bladel1992singular}.
Nevertheless, we can derive stronger bounds than those implied by \Cref{lem: variable bounds} by once again invoking the conservation of energy, and adding the physically-motivated assumption that the antenna we optimize has some finite electrical conductivity $\sigma > 0$.
After all, a current on a real material loses some energy to heat, and the greater the current, the more heat is produced.
(This is known as \emph{Ohmic} or \emph{Joule heating}~\cite{pozar2012microwave}.)

Since we approximate the current density by a weighted sum of basis functions, each supported on a nonzero area, we can find bounds on the current density coefficients $I_f$ for each basis function.
Let $f \in B_\M$ be a basis function with associated expansion coefficient $I_f$.
For an ohmic material, $J = \sigma E$, 

For a current density $J: \Sigma \to \complex^3$ on a metallic surface $\Sigma$ with surface resistance $R_s$, the power lost to heat is~\cite[Eq.\ 1.131]{pozar2012microwave}
\[
    Q = \frac{R_s}{2}\int_\Sigma |J(x)|^2 \dd x.
\]
The surface resistance $R_s$ can be computed to great accuracy for good conductors.
For a conductor with conductivity $\sigma$, a time-harmonic current of frequency $f$ experiences a sheet resistance of~\cite[Eq.\ 1.125]{pozar2012microwave}
\[
    R_s = \sqrt{\frac{2\pi \mu_0 f}{2\sigma}}.
\]
While an approximation, this formula is very accurate in practice.

If the current density is given as a linear combination of basis functions $f \in B_\M$ with expansion coefficients $I \in \C^{B_\M}$, then
\[
    Q = \frac{R_s}{2}\sum_{f, g \in B_{\M}} \overline{I_f}I_g \int_\Sigma f(x) \cdot g(x) \dd x = \frac{R_s}{2} \cdot I^H G I,
\]
where
\[
    G_{fg} =  \int_\Sigma f(x) \cdot g(x) \dd x, \quad f, g \in B_\M.
\]
This matrix $G$ is a Gram matrix and therefore positive semidefinite.

By the law of conservation of energy, the power dissipated as heat cannot exceed the input power.
Therefore, if $P_\mathrm{max} > 0$ is an upper bound to the input power, the expansion coefficients $I$ satisfy
\[
    I^H G I \leq 2P_{\max}/R_s.
\]
Since $G$ is positive semidefinite, we can solve the following problem for each $I_f$ to find bounds on these coefficients:
\begin{maxi*}|s|
    {I\in \C^{B_\M}}{|I_f|}{}{}
    \addConstraint{I^H G I}{\leq 2P_{\max}/R_s}{}.
\end{maxi*}

Adding additional (valid) constraints to this program can improve the variable bounds.
However, for our purposes this simple program already suffices.
Similarly to the procedure of \Cref{lem: variable bounds}, we then obtain the tight bounds
\[
    |I_f| \leq \sqrt{\frac{2P_\mathrm{max}(G^{-1})_{ff}}{R_s}}.
\]
This procedure still produces unrealistically high bounds (in the range of 168-645 A/m for the patch antenna that we consider in \Cref{section: computation}), but they are small enough that Program \eqref{eq: base mip} can be effectively reformulated using big-$M$ constraints.

\section{Constraints from Design Specifications}
\label{section: constraints}

The MIP derived in \Cref{thm: base mip} is an exact formulation for finding the sub-mesh $\M' \subseteq \M$ that maximizes the far-field amplitude in a fixed direction, subject to a fixed incident electric field.
In practical cases however, various other design considerations must be taken into account to produce a useful antenna design.
We outline three such design considerations: Geometric constraints, impedance constraints and polarization constraints.

Geometric constraints are easily realizable as linear constraints on a graph representation of $\M$.
Impedance and polarization constraints, on the other hand, require some care: Implemented naively, they introduce again some non-convexity into the problem, but with a similar approach as in \Cref{thm: base mip} we can formulate them exactly as linear constraints.

\subsection{Geometric Constraints}
\label{sec: geometric constraints}

In~\cite{hoeksma2026solving}, we introduced a simple set of constraints to ensure that the selected mesh $\M' \subseteq \M$ forms a connected region. 
To this aim, we constructed the graph $G_\M$ with vertex set $V(G_\M) = T$ and edge set
\[
    E(G_\M) = \{\{t_i, t_j\} : \text{$t_i$ and $t_j$ share an edge in $\M$}\}.
\]
Connectivity of a sub-mesh $\M'$ can then be ensured through a set of flow constraints on $G_\M$.
Moreover, $G_\M$ gives us a great deal of flexibility to impose additional geometric constraints on Program \eqref{eq: base mip}.

One particularly important geometric constraint involves point connections.
In electromagnetic simulations the physical meaning of two pieces of material meeting only in a single point is ambiguous, and should therefore be avoided~\cite{thiel2015point}.
This situation may occur due to triangle deletion, as in \Cref{fig: bad mesh}.

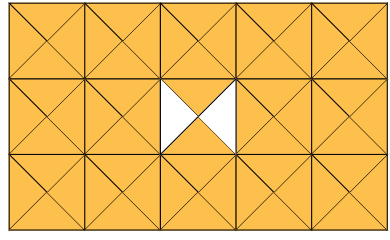
\begin{figure}
    \centering
    \begin{tikzpicture}[scale=0.5]

        \foreach \x in {-4, -2, 0, 2, 4} {
            \foreach \y in {-2, 0, 2} {
                \coordinate (c) at (\x, \y);
            
                \foreach \i in {1,...,4} {
                    \coordinate (v\i) at ({\x+sqrt(2)*cos(\i*360/4+45)},{\y+sqrt(2)*sin(\i*360/4+45)});
                }
                \foreach \i in {1,...,3} {
                    \pgfmathsetmacro{\j}{\i+1}
                    \draw[fill, pec] (v\i) -- (v\j) -- (c);
                    \draw (v\i) -- (v\j) -- (c);
                }
                \draw[fill, pec] (v4) -- (v1) -- (c);
                \draw[black] (v4) -- (v1) -- (c);
            
            }
        }
        
        \coordinate (c) at (0, 0);
        
        \foreach \i in {1,...,4} {
            \coordinate (v\i) at ({sqrt(2)*cos(\i*360/4+45)},{sqrt(2)*sin(\i*360/4+45)});
        }
        \draw[fill, white] (v1) -- (v2) -- (c);
        \draw[fill, white] (v3) -- (v4) -- (c);
        \draw[black] (v1) -- (v2) -- (c);
        \draw[black] (v3) -- (v4) -- (c);
        
    \end{tikzpicture}
    
    \caption{A mesh with an electromagnetically ambiguous point connection, where the tips of two mesh triangles meet.}
    \label{fig: bad mesh}
    
\end{figure}

For a mesh $\M' \subseteq \M$, we call a vertex $v$ of $\M$ a \emph{point connection} if $v$ the sub-mesh generated by keeping only triangles that contain $v$ is disconnected.
This corresponds exactly to the situation in \Cref{fig: bad mesh}.
Let $G_\M$ be as before.
If $x$ is an integer feasible solution of $P_1$ of \Cref{thm: base mip}, then a vertex $v$ of $\M$ is a point connection if and only if and only if the subgraph of $G_\M$ induced by those triangles that contain $v$ consists of at least two disconnected components.
Therefore, the following geometric constraints exclude point connections:
\begin{quote}
    For each vertex $v \in \V(\M)$, the subgraph of $G_\M$ induced by those $t \in T(\M)$ that contain $v$ is connected.
\end{quote}
Let $d_\M(v)$ be the degree of $v \in \V(\M)$ in $\M$.
Then this constraint can be implemented using $O\left(\sum_{v \in \V(\M)} d_\M(v)\right)$ linear constraints.

We note that a similar formulation was used by Neuman et al.~\cite{neuman2026towards}.
However, they use a nonlinear graph-algebraic parameter as a penalty term in a genetic algorithm.
In contrast, our formulation is linear and exact.

\subsection{Impedance}
\label{sec: impedance}
In~\cite{hoeksma2026solving}, we introduced impedance constraints in our model. Recall that the input impedance of an antenna is defined as 
\begin{align}
\label{eq: impedance definition}
    Z_\mathrm{in} := \frac{V_\mathrm{in}}{I_\mathrm{feed}},
\end{align}
and the feed current $I_\mathrm{feed}$ can be found as a weighted sum of current coefficients:
\[
    I_\mathrm{feed} = c_\mathrm{feed}^T I.
\]
Recall also that $Z_{\mathrm{in}}$ is a complex number, and that we ideally wish for $Z_{\mathrm{in}}$ to equal exactly $50 \, \Omega$.

In~\cite{hoeksma2026solving}, we introduced a method to constrain the magnitude of the input impedance through the single second-order cone constraint 
\begin{align}
\label{eq: impedance lower bound}    
    |Z_\mathrm{in}| \geq |Z_0| \iff |I_\mathrm{feed}| \leq \frac{V_\mathrm{in}}{|Z_0|}.
\end{align}
While we have found this constraint to be sufficient in many practical use cases, we have also encountered cases where this constraint is not sufficient to guarantee a well-matched antenna, as it allows for a large imaginary and small real part of the impedance. 

Therefore, we introduce a way to separately constrain the real and imaginary part of the impedance.
This requires some care.
Let
\[
    R_\mathrm{in} = \Re\{Z_\mathrm{in}\}, \quad X_\mathrm{in} = \Im\{Z_\mathrm{in}\}
\]
denote the \emph{input resistance} and the \emph{input reactance}, respectively.
Then naively, one could add the constraints
\begin{align*}
    R_\mathrm{min} &\leq R_\mathrm{in} \leq R_\mathrm{max}, \\
    X_\mathrm{min} &\leq X_\mathrm{in} \leq X_\mathrm{max}.
\end{align*}
This is however non-convex in the current variables $I$, since $I_\mathrm{feed} = c_\mathrm{feed}^T I$.

Notice that, if $\Im\{I_\mathrm{feed}\} = 0$, then \Cref{eq: impedance definition} can be rewritten as
\begin{align*}
    R_\mathrm{in} \Re\{I_\mathrm{feed}\} +  \jim X_\mathrm{in} \Re\{I_\mathrm{feed}\} = V_\mathrm{in}.
\end{align*}
This then decomposes into two separate linear equations---the real and imaginary parts of the above---which involve only $R_\mathrm{in}$ or $X_\mathrm{in}$, allowing us to formulate linear constraint equations on these quantities in terms of $I_\mathrm{feed}$.
We can then use a similar trick as in~\cite{hoeksma2026solving} to find a solution where $I_\mathrm{feed}$ is real.

\begin{theorem}
\label{lem: impedance rectangular}
    Let $c_\mathrm{feed} \in \complex^{B_\M}$ and $V_\mathrm{feed} > 0$, and let $X_\mathrm{min} \leq X_\mathrm{max}$ and $0 < R_\mathrm{min} \leq R_\mathrm{max}$.
    If Program \eqref{eq: base mip} is augmented with the auxiliary variables $J \in \complex^{B_\M}$ and $a, b \in \real$, and the constraints
    %
    \begin{align*}
        \label{eq: impedance rectangular}
        \tag{I1}
        y_f = 1 \implies &Z_{f,*} J = (a + \jim b)V_f, \\
        y_f = 0 \implies &J_f = 0, \nonumber \\
        \Im\{c_\mathrm{feed}^T J\} &= 0, \nonumber \\
        \Re\{c_\mathrm{feed}^T J\} &= 1, \nonumber \\
        R_\mathrm{min} \leq V_\mathrm{in} a &\leq R_\mathrm{max} , \nonumber \\
        X_\mathrm{min} \leq V_\mathrm{in} b &\leq X_\mathrm{max}, \nonumber
    \end{align*}
    then any integer feasible solution $(x, I, z)$ of the resulting program satisfies
    \begin{align*}
        \tag{Z}
        R_\mathrm{min} \leq \Re\left\{\frac{V_\mathrm{in} z}{c_\mathrm{feed}^T I}\right\} \leq R_\mathrm{max} \quad \text{and} \quad
        X_\mathrm{min} \leq \Im\left\{\frac{V_\mathrm{in} z}{c_\mathrm{feed}^T I}\right\} \leq X_\mathrm{max}.
    \end{align*}
    Conversely, if an integer feasible solution $(x, I)$ of Program \eqref{eq: base mip} satisfies (Z), then there exist $a, b \in \complex$ and $J \in \complex^{B_\M}$ that satisfy the constraints (\ref{eq: impedance rectangular}).
\end{theorem}

\begin{proof}
    Let $(I, J, a, b, z)$ be feasible for Program \eqref{eq: base mip} with the added constraints from the lemma statement. 
    First, we argue that $a \neq 0$ in any feasible solution.
    Indeed, if $a = 0$, then
    \[
        0 < R_\mathrm{min} \leq V_\mathrm{in} a = 0,
    \]
    a contradiction.
    This implies that $a + \jim b \neq 0$.
    Now, we have $Z_\M I = V_\M z$ and $Z_\M J = (a + \jim b) V_\M$, and so $I = J z / (a + \jim b)$.
    Therefore,
    \[
        \frac{V_\mathrm{in}}{c^T_\mathrm{feed} I} = \frac{(a+\jim b)V_\mathrm{in}}{c^T_\mathrm{feed}Jz} = \frac{(a + \jim b)V_\mathrm{in}}{z},
    \]
    since $\Re\{c_\mathrm{feed}^T J\} = 1$ and $\Im\{c_\mathrm{feed}^T J\} = 0$.
    Then
    \begin{align*}
        \Re\left\{\frac{V_\mathrm{in}z}{c^T_\mathrm{feed} I}\right\} 
        = \Re\left\{(a+\jim b)V_\mathrm{in}\right\}
        = a V_\mathrm{in}
        \intertext{and}
        \Im\left\{\frac{V_\mathrm{in}z}{c^T_\mathrm{feed} I}\right\} 
        = \Im\left\{(a+\jim b)V_\mathrm{in}\right\}
        = b V_\mathrm{in}.
    \end{align*}
    Since any feasible solution satisfies $R_\mathrm{min} \leq a V_\mathrm{in} \leq R_\mathrm{max}$ and $X_\mathrm{min} \leq b V_\mathrm{in} \leq X_\mathrm{max}$, one direction is proved.

    For the converse direction, suppose a solution $(x, I)$ of Program \eqref{eq: base mip} satisfies
    \begin{align*}
        R_\mathrm{min} \leq \Re\left\{\frac{V_\mathrm{in}z}{c_\mathrm{feed}^T I}\right\} \leq R_\mathrm{max} \quad \text{and} \quad
        X_\mathrm{min} \leq \Im\left\{\frac{V_\mathrm{in}z}{c_\mathrm{feed}^T I}\right\} \leq X_\mathrm{max}.
    \end{align*}
    Let $I_f = c^T_\mathrm{feed} I$.
    The above inequalities imply that $I_f \neq 0$.
    Then, let $a + \jim b = \frac{1}{I_f}$, and define $J = (a + \jim b)I$.
    Then clearly $c_\mathrm{feed}^T J = 1$ by construction, $J$ satisfies the first two constraints of \Cref{eq: impedance rectangular}, and
    \[
        V_\mathrm{in}a = V_\mathrm{in}\Re\left\{\frac{c_\mathrm{feed}^T J}{c_\mathrm{feed}^T I}\right\}
         = \Re\left\{
            \frac{V_\mathrm{in}}{c_\mathrm{feed}^T I}
         \right\}.
    \]
    Therefore, $a$ satisfies $R_\mathrm{min} \leq V_\mathrm{in} a \leq R_\mathrm{max}$.
    A similar analysis shows $b$ also satisfies $X_\mathrm{min} \leq V_\mathrm{in}b \leq X_\mathrm{max}$.
\end{proof}

\Cref{lem: impedance rectangular} shows that, despite its non-linearity in the current variables $I$, the input impedance $Z_\mathrm{in}$ can be constrained using purely linear inequalities---although at the cost of introducing $|B_\M|+2$ auxiliary variables.

These constraints are exact and reasonably efficient, but they are not the unique linear approach.
The following theorem shows an alternative formulation of the impedance constraints, which we find performs slightly better in practice.

\begin{theorem}
\label{lem: impedance wedge}
    Let $c_\mathrm{feed}, V_\mathrm{feed}, X_\mathrm{min}, X_\mathrm{max}, R_\mathrm{min}$ and  $R_\mathrm{max}$ be as in \Cref{lem: impedance rectangular}.
    If Program \eqref{eq: base mip} is augmented with the auxiliary variables $J \in \complex^{B_\M}$, $a, b \in \real$ and $u \in \{0, 1\}$, and the constraints
    \begin{align*}
        \tag{I2}
        \label{eq: impedance wedge}
        y_f = 1 \implies &Z_{f,*} J = (a + \jim b)V_f, \\
        y_f = 0 \implies &J_f = 0, \nonumber \\
        u = 1 \implies &(a + b = 1) \wedge (a \geq 0, b > 0), \\
        u = 0 \implies &(a - b = 1) \wedge (a \geq 0, b \leq 0), \\
        \Im\{c_\mathrm{feed}^T J\} &= 0, \\
        R_\mathrm{min} \Re\{c_\mathrm{feed}^T J\} &\leq V_\mathrm{in} a \leq R_\mathrm{max} \Re\{c_\mathrm{feed}^T J\}, \tag{R} \\
        X_\mathrm{min} \Re\{c_\mathrm{feed}^T J\} &\leq V_\mathrm{in} b \leq X_\mathrm{max} \Re\{c_\mathrm{feed}^T J\} \tag{X}, 
    \end{align*}
    then any integer feasible solution $(x, I, z)$ of the resulting program satisfies
    \begin{align*}
        \tag{Z}
        R_\mathrm{min} \leq \Re\left\{\frac{V_\mathrm{in}z}{c_\mathrm{feed}^T I}\right\} \leq R_\mathrm{max} \quad \text{and} \quad
        X_\mathrm{min} \leq \Im\left\{\frac{V_\mathrm{in}z}{c_\mathrm{feed}^T I}\right\} \leq X_\mathrm{max}.
    \end{align*}
    Conversely, if an integer feasible solution $(x, I)$ of the MIP satisfies (Z), then there exist $a, b \in \complex$ and $J \in \complex^{B_\M}$ that satisfy the constraints (\ref{eq: impedance wedge}).
\end{theorem}

\begin{proof}
    We start by noting that the indicator constraints above imply that $a + \jim b \neq 0$, and so we can immediately write $I = Jz / (a + \jim b)$.
    Note that, as in~\cite{hoeksma2026solving}, feasibility of the antenna optimization problem implies that $z \neq 0$.
    
    Next, observe that constraint (R) implies that $\Re\{c_\mathrm{feed}^T J\} \geq 0$ for any feasible solution, since $V_\mathrm{in} > 0$, $R_\mathrm{max} > 0$, and $a \geq 0$.
    Indeed, we can show that for any feasible solution it holds that $\Re\{c_\mathrm{feed}^T J\} > 0$.
    We distinguish three cases:
    \begin{description}
        \item[Case 1: $a > 0$.] Then $R_\mathrm{max} \Re\{c_\mathrm{feed}^T J\} \geq V_\mathrm{in}a > 0$.
        Since $R_\mathrm{max} > 0$ by assumption, it follows that $\Re\{c_\mathrm{feed}^T J\} > 0$.
        \item[Case 2: $a = 0$ and $u = 1$.] Then $b = 1$, which implies that 
        \[
            V_\mathrm{in} \leq X_\mathrm{max} \Re\{c_\mathrm{feed}^T J\},
        \]
        and since $V_\mathrm{in} > 0$ we have $\Re\{c_\mathrm{feed}^T J\} \neq 0$.
        \item[Case 3: $a = 0$ and $u = 0$.] Then $b = -1$, which implies that 
        \[
            X_\mathrm{min} \Re\{c_\mathrm{feed}^T J\} \leq -V_\mathrm{in},
        \]
        and since $V_\mathrm{in} > 0$ we have $\Re\{c_\mathrm{feed}^T J\} \neq 0$.
    \end{description}

    Now that we know $\Re\{c_\mathrm{feed}^T J\} \neq 0$, and therefore also $\Re\{c_\mathrm{feed}^T I\} \neq 0$, we can write
    \[
        \frac{V_\mathrm{in}z}{c^T_\mathrm{feed} I} = \frac{(a+\jim b)V_\mathrm{in}}{c^T_\mathrm{feed}J}.
    \]
    Since $V_\mathrm{in} \in \real$ and $\Im\{c_\mathrm{feed}^T J\} = 0$ in any feasible solution, we then have
    \begin{align}
    \label{eq: transform I to J}    
        \Re\left\{\frac{V_\mathrm{in}z}{c_\mathrm{feed}^T I}\right\} = \frac{a V_\mathrm{in}}{\Re\{c^T_\mathrm{feed} J\}}
        \quad \text{and} \quad
        \Im\left\{\frac{V_\mathrm{in}z}{c_\mathrm{feed}^T I}\right\} = \frac{b V_\mathrm{in}}{\Re\{c^T_\mathrm{feed} J\}}.
    \end{align}
    Since  $\Re\{c_\mathrm{feed}^T J \} > 0$, we can divide (R) and (X) through by this quantity and find
    \[
        R_\mathrm{min} \leq \frac{a V_\mathrm{in}}{\Re\{c_\mathrm{feed}^T J\}} \leq R_\mathrm{max} \quad \text{and} \quad
        X_\mathrm{min} \leq \frac{b V_\mathrm{in}}{\Re\{c_\mathrm{feed}^T J\}} \leq X_\mathrm{max},
    \]
    which is exactly the desired result after inserting \Cref{eq: transform I to J}.

    Now for the converse direction, assume that $(x, I, z)$ is feasible for Program \eqref{eq: base mip} and satisfies (Z).
    Let $I_f = c^T_\mathrm{feed} I$.
    Let $u \in \{0, 1\}$.
    Draw a line in the complex plane from $0$ to $I_f/z$.
    If $\Im\{I_f/z\} \geq 0$, then we let $a - \jim b$ be the intersection of this line with the line segment from $\jim$ to $1$, and we set $u = 0$.
    Otherwise, we let $a - \jim b$ be the intersection of this line with the line segment from $-\jim$ to $1$, and we set $u = 1$.
    By construction we now have $a - b = 1$ if $u = 0$, and $a + b = 1$ if $u = 1$.

    We define $J = (a+\jim b)I/z$ and let $J_\mathrm{feed} = c_\mathrm{feed}^T J$.
    Assume $u = 0$.
    Note that by construction, $a - \jim b = r \cdot I_\mathrm{feed}/z$ for some positive $r \in \real$.
    Thus, we have
    \[
        J_\mathrm{feed} = (a + \jim b)I_\mathrm{feed}/z = (a + \jim b) \cdot \frac{1}{r} \cdot (a - \jim b) \cdot z/z = \frac{a^2 + b^2}{r} \in \real_+.
    \]
    In other words, $c_\mathrm{feed}^T J$ is real and non-negative, and so $\Im\{c_\mathrm{feed}^T I\} = 0$ as required.

    To show that $a$ and $b$ satisfy (R) and (X) we compute:
    \[
        V_\mathrm{in}a = V_\mathrm{in}\Re\left\{
            \frac{c_\mathrm{feed}^T Jz}{c_\mathrm{feed}^T I}
        \right\}
            = \Re\{c_\mathrm{feed}^T J\} \cdot  \Re\left\{
                \frac{V_\mathrm{in}z}{c_\mathrm{feed}^T I}
            \right\}.
    \]
    We have already shown that $\Re\{c_\mathrm{feed}^T J\} \geq 0$, and so the fact that $(J, a, b)$ satisfies (R) and (X) follows from (Z).
    For the case $u = 1$ the analysis is identical.
\end{proof}

As \Cref{lem: impedance wedge} yields $n + 3$ extra variables and two indicator constraints, one might expect that a program with these constraints performs worse than a program applying \Cref{lem: impedance rectangular}; however, this is not the case in our experiments.
We therefore provide both formulations here, since it is likely solver-dependent which constraint formulation outperforms which.

\begin{remark}
    To the best of our knowledge this is the first time that constraints on the input impedance of an antenna have been formulated as a set of purely linear constraints.
    Previous attempts at antenna topology optimization typically directly implement \Cref{eq: impedance definition}, either as a penalty term for a genetic algorithm or as a hard nonlinear constraint~\cite{vanbeurden2003integro}.
    We expect that the latter approach also benefits from using \Cref{lem: impedance rectangular} or \Cref{lem: impedance wedge}.
\end{remark}

Finally, we remark that imposing $R_\mathrm{in} \geq R_\mathrm{min}$ (or even $|Z_{\mathrm{in}}| \leq 50 \Omega$ immediately yields an upper bound $P_\mathrm{max} = V^2_\mathrm{in}/R_\mathrm{min}$ (or $V^2_{\mathrm{in}}/|Z_{\mathrm{in}}|$) on the input power.
It is a standard result from antenna theory that, for a perfectly electrically conducting antenna, the input resistance is related to the radiated power as~\cite[Section 2.13]{balanisAntennaTheoryAnalysis2012}
\begin{align}
\label{eq: pmax}
    P_\mathrm{rad} = \frac{V_\mathrm{in}^2}{|Z_{\mathrm{in}}|}.
\end{align}
Thus, $R_\mathrm{in} \geq R_\mathrm{min}$ implies $P_\mathrm{rad} \leq V_\mathrm{in}^2/R_\mathrm{min} = P_\mathrm{max}$.
Imposing impedance constraints is thus not only a practical necessity, but also yields valid bounds for \Cref{lem: energy conservation}, \Cref{lem: variable bounds} and \Cref{cor: dual bound}.

\subsection{Polarization}
\label{sec: polarization}

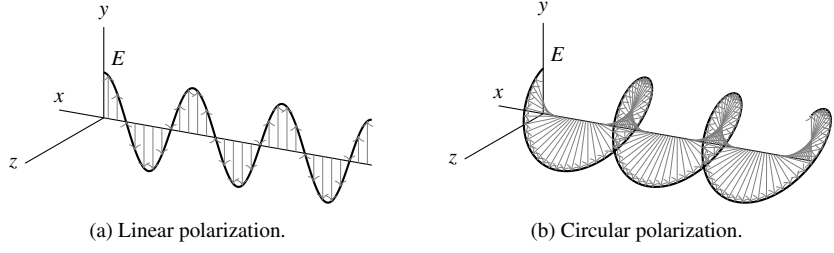
\begin{figure}
    \centering
    \begin{subfigure}{0.5\textwidth}
    \centering
        \begin{tikzpicture}[x={(-10:1cm)},y={(90:1cm)},z={(210:1cm)},scale=0.6]
        \draw (-1,0,0) node[above] {$x$} -- (6,0,0);
        \draw (0,0,0) -- (0,2,0) node[above] {$y$};
        \draw (0,0,0) -- (0,0,2) node[left] {$z$};
        \draw[thick] plot[domain=0:6,samples=200] (\x,{cos(deg(pi*\x))},0);
        \foreach \x in {0.1,0.3,...,6} {
          \draw[->,help lines] (\x,0,0) -- (\x,{cos(deg(pi*\x))},0);
        }
        \node[above right] at (0,1,0) {$E$};
    \end{tikzpicture}
    \caption{Linear polarization.}
    \end{subfigure}\hfill
    \begin{subfigure}{0.5\textwidth}
    \centering
        \begin{tikzpicture}[x={(-10:1cm)},y={(90:1cm)},z={(210:1cm)},scale=0.6]
        \draw (-1,0,0) node[above] {$x$} -- (6,0,0);
        \draw (0,0,0) -- (0,2,0) node[above] {$y$};
        \draw (0,0,0) -- (0,0,2) node[left] {$z$};
        \draw[thick] plot[domain=0:6,samples=200] (\x,{cos(deg(pi*\x))},{sin(deg(pi*\x))});
        \foreach \x in {0.1,0.15,...,6} {
          \draw[->,help lines] (\x,0,0) -- (\x,{cos(deg(pi*\x))},{sin(deg(pi*\x))});
        }
        \node[above right] at (0,1,0) {$E$};
    \end{tikzpicture}
    \caption{Circular polarization.}
    \end{subfigure}
    \caption{
        The figure traced out by the electric field vector for traveling electromagnetic waves of linear and circular polarization. 
        The waves depicted travel in the positive $x$-direction.
        \label{fig: polarization}
    }
\end{figure}

Since electromagnetic radiation is a \emph{transversal} wave, one can speak of the \emph{polarization} of the radiation emitted by an antenna.
Simply put, the polarization of a wave is determined by the figure traced out by its amplitude vector as the wave propagates through space when facing the propagation direction.
If this figure is a line segment, then the wave is \emph{linearly polarized}; if the figure is a circle, then one speaks of \emph{circular polarization}.
See \Cref{fig: polarization} for a schematic depiction.

The polarization characteristics of an antenna are important for various practical purposes.
For instance, a strongly linearly polarized antenna can be used to reduce interference from spurious sources, by matching the orientation of a transmitting antenna to its intended receiver.
Conversely, one may wish to design an antenna that is less sensitive to its spatial orientation; in this case, a circularly polarized antenna is appropriate~\cite{balanisAntennaTheoryAnalysis2012}.

To quantify the degree to which an antenna with associated far-field pattern $\EFF$ is polarized, we define its \emph{polarization purity}.

\begin{definition}[Polarization purity]
\label{def: polarization}
    Fix a unit vector $\hat{r}_{\vartheta, \varphi} \in \mathbb{S}^2$ in the $(\vartheta, \varphi)$-direction and two mutually orthogonal unit vectors $e_x$ and $e_y$ with $e_x \cdot \hat{r}_{\vartheta,\varphi} = e_y \cdot \hat{r}_{\vartheta,\varphi} = 0$.
    We say that an antenna has $\epsilon$-pure linear polarization in the $(\vartheta,\varphi)$-direction if the far-field $E_\mathrm{FF}$ radiated by the antenna satisfies
    \[
        |e_y \cdot \EFF(\vartheta, \varphi)| \leq \epsilon \cdot |e_x \cdot \EFF(\vartheta, \varphi)|.
    \]
    Let $e_{\vartheta}$ denote the unit vector in the $\vartheta$-direction and $e_{\varphi}$ the unit vector in the $\varphi$-direction.
    We say that an antenna has $\epsilon$-pure circular polarization in the $(\vartheta, \varphi)$-direction if
    \begin{align*}
        |e_\vartheta \cdot E_\mathrm{FF}(\vartheta, \varphi) - \jim e_\varphi \cdot E_\mathrm{FF}(\vartheta, \varphi)| &\leq \epsilon \cdot |e_\vartheta \cdot E_\mathrm{FF}(\vartheta, \varphi)|.
    \end{align*}
\end{definition}

Thus far, we have only concerned ourselves with maximizing some fixed vector component of the electric field radiated by an antenna, without regard to the components of the field in other directions.
For many applications however, some control on the polarization of the antenna is required~\cite{balanisAntennaTheoryAnalysis2012}.
\Cref{def: polarization} suggests a simple implementation in terms of linear constraints, when the direction $(\vartheta, \varphi)$ of \Cref{def: polarization} is identical to the direction in which we maximize the field.

\begin{lemma}\label{lem: polarization}
    Fix $(\vartheta, \varphi)$ and two mutually orthogonal unit vectors $e_x$ and $e_y$ with $e_x \cdot \hat{r}_{\vartheta,\varphi} = e_y \cdot \hat{r}_{\vartheta,\varphi} = 0$.
    Assume that $e_x = e_\vartheta$.
    Let the coefficients $c_\M$ satisfy for a feasible solution $(\M, I)$ of the program of \Cref{thm: base mip}
    \[
        c_\M^T I = e_x \cdot \EFF(\vartheta, \varphi).
    \]
    If the antenna optimization program of \Cref{thm: base mip} with objective $\Re\{c_\M^T I\}$ is augmented with the constraints 
    \begin{align*}
        |e_y \cdot \EFF(\vartheta, \varphi)| \leq \epsilon \cdot e_x \cdot \EFF(\vartheta, \varphi)
    \end{align*}
    then the optimal solution $\M^*$ of the resulting program corresponds to an $\epsilon$-pure linearly polarized antenna.
    Similarly, if the program is augmented with
    \begin{align*}
        |e_\vartheta \cdot E_\mathrm{FF}(\vartheta, \varphi) - \jim e_\varphi \cdot E_\mathrm{FF}(\vartheta, \varphi)| &\leq \epsilon \cdot e_\vartheta \cdot E_\mathrm{FF}(\vartheta, \varphi)
    \end{align*}
    then the optimal solution $\M^*$ of the resulting program corresponds to an $\epsilon$-pure circularly polarized antenna.
\end{lemma}

\begin{proof}
    By \Cref{cor: objective real}, we have $\Im \{ c^T_\M I^*\} = 0$ and $c^T_\M I^* \geq 0$ for any optimal solution $(\M^*, I^*)$.
    Thus, $e_x \cdot \EFF(\vartheta, \varphi) \geq 0$ for $\EFF$ computed from any optimal solution.
    One can therefore drop the absolute value operators from the right-hand side of the inequalities in \Cref{def: polarization} for any optimal solution.
    The result is an antenna whose current density $I^*$ produces a field $\EFF$ that satisfies $\epsilon$-pure linear polarization.
    The same argument holds for circular polarization, since $e_x = e_\vartheta$ by assumption.
\end{proof}

\section{Computational Results}
\label{section: computation}

We will showcase our improvements by optimizing a patch antenna.
For this test case, the base mesh $\M$ was generated using the open-source program Gmsh~\cite{geuzaine2009gmsh}.
The electromagnetics simulations, necessary to obtain the matrices $Z_M$, were performed using the open-source boundary element method library Bempp-cl~\cite{betcke2021bempp}.
All optimization was performed using Gurobi 11.0.2~\cite{gurobi} on a  system with  32 Intel Xeon Gold 5217 cores and 16 GB of memory.

\subsection{Patch Antenna}
\label{section: patch}

Patch antennas are perhaps the most important type of antenna for many modern use cases, with applications in radar systems, mobile communications, medical imaging, and many other fields~\cite{balanisAntennaTheoryAnalysis2012}.
They consist of a flat strip of metal suspended above a much larger metal plate, the \emph{ground plane}.
The electromagnetic interactions between the ground plane and the antenna direct the radiated fields away from the antenna.

Patch antennas admit a great deal of flexibility in their design.
As such, they come in various basic shapes---rectangles, circles, semi-circles, and polygons have all been proposed for various purposes.
On top of this, the topology of patch antennas can be altered to achieve various goals, including a stronger radiation response and good polarization characteristics~\cite{balanisAntennaTheoryAnalysis2012}.
This design freedom makes the patch antenna a perfect testing ground for our MIP.

\paragraph*{Mesh Description.}~
The mesh $\M$ we use is displayed in a top-down view in \Cref{fig: patch}.
The main body of the mesh is a square of side-length $L = 0.45\lambda$, where $\lambda \approx 10\, \mathrm{cm}$ is the wavelength of electromagnetic radiation at a frequency $f = 3\, \mathrm{GHz}$.
Protruding from the side of the mesh is the \emph{feed line}, which is a narrow rectangular strip of width $L/10$.
The feed line extends a distance $L/4$ towards the negative-$x$ direction and subsequently extends down towards the ground plane at $z = 0$.
The vertical distance from the patch to the (infinite) ground plane is $0.0425\lambda$.
The entire mesh, including the feed line, contains $91$ triangles.

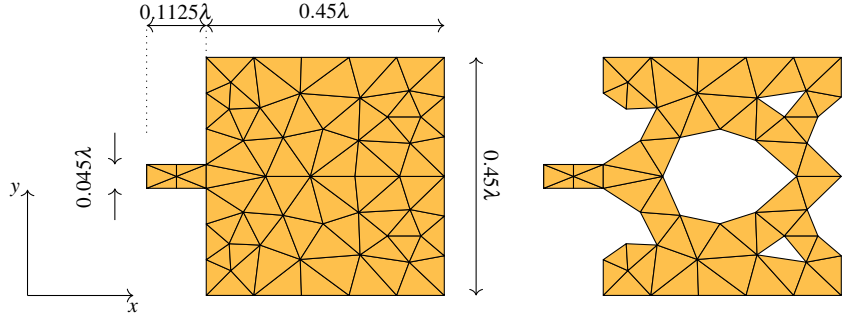
\begin{figure}
    \centering
    \begin{tikzpicture}
    [
      scale=0.7,
      x={(1cm,0cm)},
      y={(0cm,1cm)},
      z={(0cm,0cm)}
    ]
        \pgfmathsetmacro{\lam}{10} 
        \pgfmathsetmacro{\W}{\lam*0.45}
        \pgfmathsetmacro{\offset}{0.6}
        \pgfmathsetmacro{\WP}{\W/10}
        \pgfmathsetmacro{\LP}{\W/4}

        \draw[<->] (\W/2+\offset, \W/2) -- (\W/2 + \offset, -\W/2) node[above, sloped, midway] {$0.45\lambda$}; 
        \draw[<->] (-\W/2, \W/2+\offset) -- (\W/2, \W/2 + \offset) node[above, sloped, midway] {$0.45\lambda$}; 

        \draw[->] (-\W/2-\LP-\offset, -\WP/2-\offset) -- (-\W/2-\LP-\offset, -\WP/2); 
        \draw[->] (-\W/2-\LP-\offset, \WP/2+\offset) -- (-\W/2-\LP-\offset, \WP/2);
        \node[rotate=90] () at (-\W/2-\LP-2*\offset, 0) {$0.045\lambda$};

        \draw[dotted] (-\W/2-\LP, \WP/2+\offset) -- (-\W/2-\LP, \W/2+\offset);
        \draw[dotted] (-\W/2, \W/2) -- (-\W/2, \W/2+\offset);
        \draw[<->] (-\W/2-\LP, \W/2+\offset) -- (-\W/2, \W/2+\offset) node[above, sloped, midway, yshift=-\offset] {$0.1125\lambda$};

        \draw[->] (-\W*1.25, -\W*0.5) -- ++ (2, 0) node[below] {$x$};
        \draw[->] (-\W*1.25, -\W*0.5) -- ++ (0, 2) node[left] {$y$};

        \tikzset{stlfacet/.style={fill=pec, draw=black, line join=round}}
        \input{square_91.tikz}

        \tikzset{shift={(\lam*0.75, 0)}}
        \input{patch_opt_91.tikz}
    \end{tikzpicture}\hspace{1em}
    \caption{\label{fig: patch}Comparison of meshes: the initial patch antenna mesh (left) and the optimized patch antenna mesh (right).}
\end{figure}

\paragraph*{Ground plane.}~
To model the ground plane, two options are available.
One is to include a ground plane in the mesh, which takes all the electromagnetic interactions between the antenna and the ground plane into account.
While this the more accurate option, it is computationally more expensive, since these interactions enter the MIP as additional constraints.
The other option is to approximate the ground plane as infinite in extent,
which introduces a boundary condition to Maxwell's equations at the ground plane.
This boundary condition in particular can be straightforwardly encoded into the matrix $Z_\M$~\cite{balanisAntennaTheoryAnalysis2012,griffithsIntroductionElectrodynamics2023}, and so this is the route we take here.

\paragraph*{Feed Model.}~
Accurate feed models in electromagnetic modelling are an active area of research.
Since this is not our focus here, we apply a simple delta-gap feed~\cite{balanisAntennaTheoryAnalysis2012,davidsonComputationalElectromagneticsRF2010}, applying a voltage difference across the antenna and its image with respect to the ground plane.
This requires us to model a feed strip which extends down from the antenna to the ground plane; see \Cref{fig: patch}.
A similar feed model can be found in the antenna simulation literature~\cite{stupfel2023efficient}.
The horizontal part of the feed-line has a width of $\lambda/20$, and extends $\lambda/8$ towards the negative-$x$ direction.

\paragraph{Objective function.}~
We maximize the $x$-directed component of the electric field that is radiated in the positive-$z$ direction (cf.\ \Cref{fig: patch}).
That is, according to the theory in \Cref{section: electromagnetism}, the objective function we maximize is
\[
    |e_x \cdot \EFF(0, 0)| = |c_\M(0, 0)^T I|.
\]

\paragraph*{Electromagnetic constraints.}~
As discussed in \Cref{sec: impedance}, it is necessary for an efficient antenna to ensure an input impedance near the reference impedance $Z_0 = 50\, \Omega$.
We found in our previous work~\cite{hoeksma2026solving} that it is often sufficient to set a lower bound on the input impedance.
Thus, we use \Cref{lem: impedance rectangular} to set
\[
    R_\mathrm{in} := \Re\{Z_{\mathrm{feed}}\} \geq 50 \, \Omega.
\]

The current density on the surface of a patch antenna has a significant component in two directions.
This means that the polarization of the radiated field must be controlled (see \Cref{sec: polarization}).
We optimize here a $0.01$-pure linearly polarized patch antenna, which means we enforce
\[
    |c_{\M, y}^T I| \leq 0.01 \cdot \Re\{c_{\M}^T I\}.
\]

\paragraph{Geometric constraints.}~
To ensure electrical connectivity and exclude electromagnetically ambiguous geometries, we include the flow-based connectivity constraints and point connection exclusion of \Cref{sec: geometric constraints}.
We also constrain the antenna to be symmetric about the $x$-axis.
(Note that the base mesh in \Cref{fig: patch} respects this symmetry.)

Finally, we add the constraint that every vertex of $G_{\M'}$ has minimum degree $2$.
This ensures that every triangle of the mesh has at least two neighbors.
In particular, this constraint excludes sharp edges which are only resolved by a single triangle.
These edges should be meshed by at least two mesh elements, since local variations in current densities can be very large (even singular) in such structures~\cite{bladel1992singular}.

\paragraph*{Current bounds.}~
We apply the current bounds described in \Cref{sec: heatingbounds}; the individual bounds range from 168 to 654 A/m.

\begin{figure}
    \centering
    \includegraphics[width=0.8\linewidth]{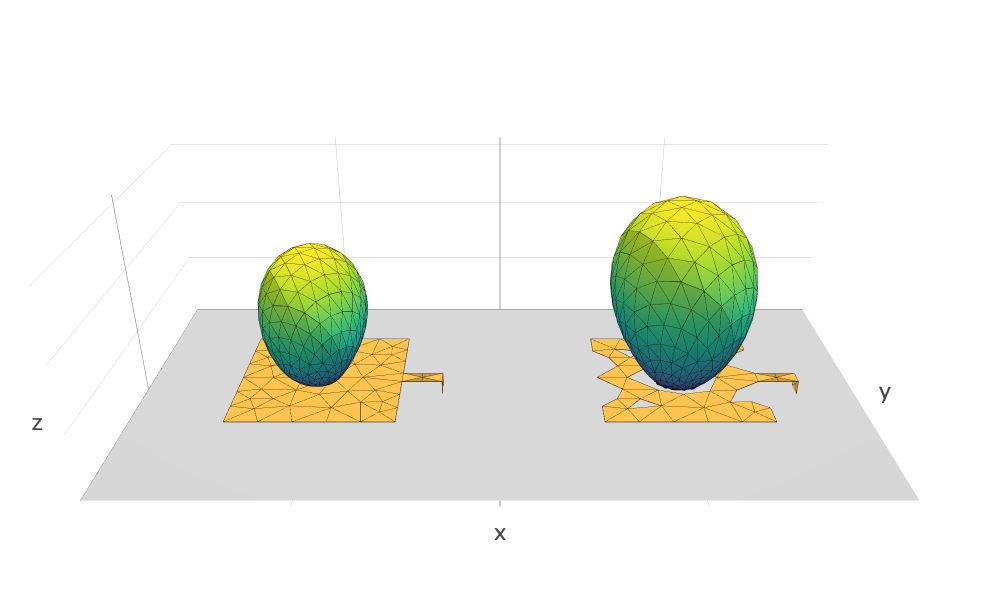}
    \caption{
        Initial and optimized radiation pattern of the patch antenna.
        The radial distance of a point on the surface from the origin corresponds the field strength in that direction.
    }
    \label{fig: patch results}
\end{figure}

\paragraph*{Results.}~
\Cref{fig: patch} (right) shows the antenna design resulting from the application of our MIP to the initial mesh $\M$ with the constraints described above.
\Cref{fig: patch results} shows the radiation pattern of the initial and optimized antennas.
The initial mesh has a far-field amplitude in the positive-$z$ direction of 1.828 V, while the optimal value is 2.421 V.
We discuss the running time in \Cref{sec: physicscuts effectiveness}.

The optimal antenna has an input impedance at $3 \, \mathrm{GHz}$ of
\[
        Z_\mathrm{in} = 50 - 3.9\jim\, \Omega.
\]
To quantify the efficiency of an antenna, a standard figure of merit in the literature is the \emph{reflection coefficient} $\Gamma$~\cite{balanisAntennaTheoryAnalysis2012}, defined as
\
\[
    \Gamma = \frac{Z - Z_0}{Z + Z_0}.
\]
One can show that the reflected power due to an impedance mismatch, as a fraction of the input power, is equal to $|\Gamma|^2$.
Filling in $Z_\mathrm{in}$ and $Z_0$ above, we have
\[
    |\Gamma| \approx 3.87 \cdot 10^{-2},
\]
implying that under ideal circumstances the optimized antenna of \Cref{fig: patch} (right) accepts $99.98\%$ of its supplied power.

It is important to note that the input impedance $Z_\mathrm{in}$ of an antenna is strongly dependent on the frequency of the input signal.
While the above demonstrates that the antenna is very well-matched at $3 \, \mathrm{GHz}$, many applications require an antenna to function over a wider range of frequencies---its \emph{bandwidth}.
The impedance bandwidth of an antenna is commonly defined~\cite{balanisAntennaTheoryAnalysis2012,davidsonComputationalElectromagneticsRF2010} as the range of frequencies over which
\[
    S_{11}(f) := 20\log_{10}|\Gamma(f)| \leq -10\, \mathrm{dB}.
\]

In \Cref{fig: reflection}, we plot $Z_\mathrm{in}(f)$ and $S_{11}(f)$ for $2.8 \, \mathrm{GHz} \leq f \leq 3.2\, \mathrm{GHz}$.
We observe that $S_{11}$ lies far below the desired threshold of $-10 \, \mathrm{dB}$ for $f = 3 \, \mathrm{GHz}$, but only stays below this threshold for a rather narrow bandwidth of approximately $78.9 \, \mathrm{MHz}$.
This is characteristic of patch antennas~\cite{balanisAntennaTheoryAnalysis2012}, so it is expected that we would achieve a very narrow-band antenna without further constraints to widen the bandwidth.

\begin{figure}
    \centering
    \includegraphics[width=\linewidth]{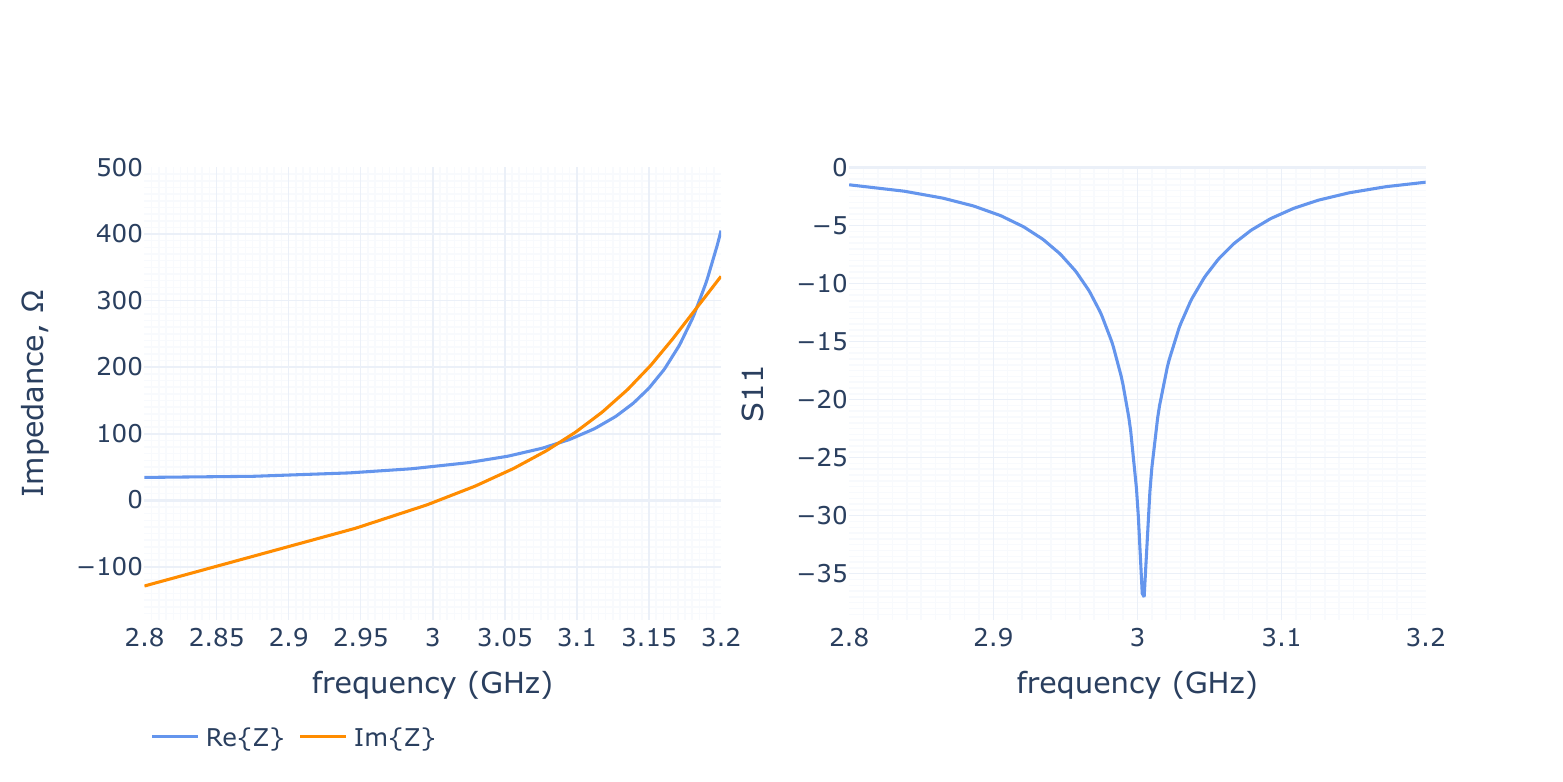}
    \caption{Input impedance and reflection coefficient of the optimized antenna for $2.8 \, \mathrm{GHz} \leq f \leq 3.2\, \mathrm{GHz}$.}
    \label{fig: reflection}
\end{figure}

\subsection{Effectiveness of Physics-Informed Cuts}
\label{sec: physicscuts effectiveness}

The physics-informed cuts and bounds of \Cref{section: physicscuts} are crucial to take the MIP from largely intractable to solvable on an ordinary computer.
In particular, the current bounds from \Cref{sec: heatingbounds} allows the solver to derive valid dual bounds, and the energy conservation constraint (\Cref{lem: energy conservation}), decomposed formulation (\Cref{thm: decomposed formulation}) and associated objective bound (\Cref{cor: dual bound}) help to lower the dual bound significantly, so that the MIP gap is reduced. 
All of these use the fact that the input power $P_\mathrm{max}$ is bounded from above by $ V_\mathrm{in}^2/R_\mathrm{min}$. 
This physical insight therefore greatly contributes to the practical feasibility of our approach.

\Cref{tab: cuts comparison} quantifies the obtained dual bound for patch antenna instances of several sizes.
Note here that the underlying patch antenna geometry is the same for all instances; only the number of triangles used to discretize the geometry changes.
Obviously, more triangles allow for more design flexibility, reflected in the increasing primal optimal value in row 7 of \Cref{tab: cuts comparison}.
On the other hand, the required computation time also increases strongly with the number of triangles.

\begin{table}[]
    \centering
    \begin{tabular}{l|l|c|c|c|c|c}
         Bound or cut                   &  Result in paper                     & $n=45$ & $n=53$ & $n=77$ & $n=85$ & $n=91$    \\ \hline
         Joule heating bounds           & \Cref{sec: heatingbounds}            & 23.16  & 179.2  & 297.0  & 330.5  & 342.0 \\
         Radiating current bounds       & \Cref{lem: variable bounds}          & 4.552  & 7.734  & 7.361  & 7.191  & 8.008  \\
         Objective cut                  & \Cref{cor: dual bound}               & 6.284  & 6.968  & 8.503  & 8.531  & 8.619 \\
         Energy constraint (full rank)  & \Cref{lem: energy conservation}      & 3.123  & 4.689  & 4.103  & 4.111  & 4.116 \\ 
         Energy constraint, $r\leq 13$  & \Cref{thm: decomposed formulation}   & 3.183  & 3.861  & 4.236  & 4.209  & 4.293 \\ \hline
         Primal optimum                 & -                                    & 1.671  & 2.076  & 2.366  & 2.403  & 2.421 \\\hline
         
         \makecell{
         Solve time \\
         (only Joule heating bounds)}
         
                                        & -                                    & 9.29 s & 10.83 s & 477 s  & 860 s & 1612 s \\ \hline
         \makecell{
         Solve time \\
         (all bounds and cuts, $r\leq 13$)}             & -                        & 5.63 s & 14.60 s   & 520 s & 1341 s & 2158 s \\  \hline
    \end{tabular}
    \caption{
        Dual bound values in the root node relaxation of each of the physics-informed bounds and cuts derived in \Cref{section: physicscuts}, for initial meshes with different numbers of triangles $n$.
        Each row of the first segment of the table applies the bounds of \Cref{sec: heatingbounds} together with \emph{only} its corresponding bound or cut.
        For the energy conservation constraint of \Cref{lem: energy conservation}, both the dual bound resulting from the full-rank matrix $R$ and its lower-rank projection on the at most $13$ most significant dimensions of its row space are given.
    }
    \label{tab: cuts comparison}
\end{table}

Including only the bounds from \Cref{sec: heatingbounds} yields a dual bound that is effectively not informative at all.
Including also the above-mentioned constraints reduces the bound to a value comparable to the primal optimal value; the energy conservation constraint is especially effective. 
We reiterate that, according to \Cref{lem: energy conservation}, \Cref{lem: variable bounds} and \Cref{cor: dual bound}, the physics-informed cuts are globally valid.
The results in \Cref{tab: cuts comparison} are thus not the effect of some cleverly-chosen approximation, but indicate that physics-informed cuts yield a much stronger formulation than naively applying \Cref{thm: base mip}.

While the energy conservation constraint is evidently very effective in reducing the dual bound, it comes at a cost.
Since the matrix $R$ from \Cref{lem: energy conservation} is generally dense and ill-conditioned, the quadratic constraint is rather expensive computationally.
One approach to reduce the computational burden is to approximate the full matrix $R$ by a low-rank surrogate, by keeping only the dimensions that correspond to the largest singular values of $R$.
This reduces the strength of the energy conservation constraint, but retains its validity.
In practice, one can drop to a rank of 10-20 while losing very little strength; see rows 5 and 6 of \Cref{tab: cuts comparison}.

\Cref{fig: dual bound} shows the progression of the dual bound for the patch antenna of \Cref{section: patch}, comparing the full-rank and low-rank energy conservation constraints.
Due to the high computational cost of the full-rank constraint, the solver can essentially stall, while with the low-rank approximation the solver is able to make slow but steady progress.
On the other hand, not including the energy conservation constraint at all still leads to a faster time to prove optimality.
If we examine the number of nodes solved over time, we see that this is due to the computational cost of the energy conservation constraints, which solve far fewer nodes per unit time.

\begin{figure}
    \centering
    \includegraphics[width=\linewidth]{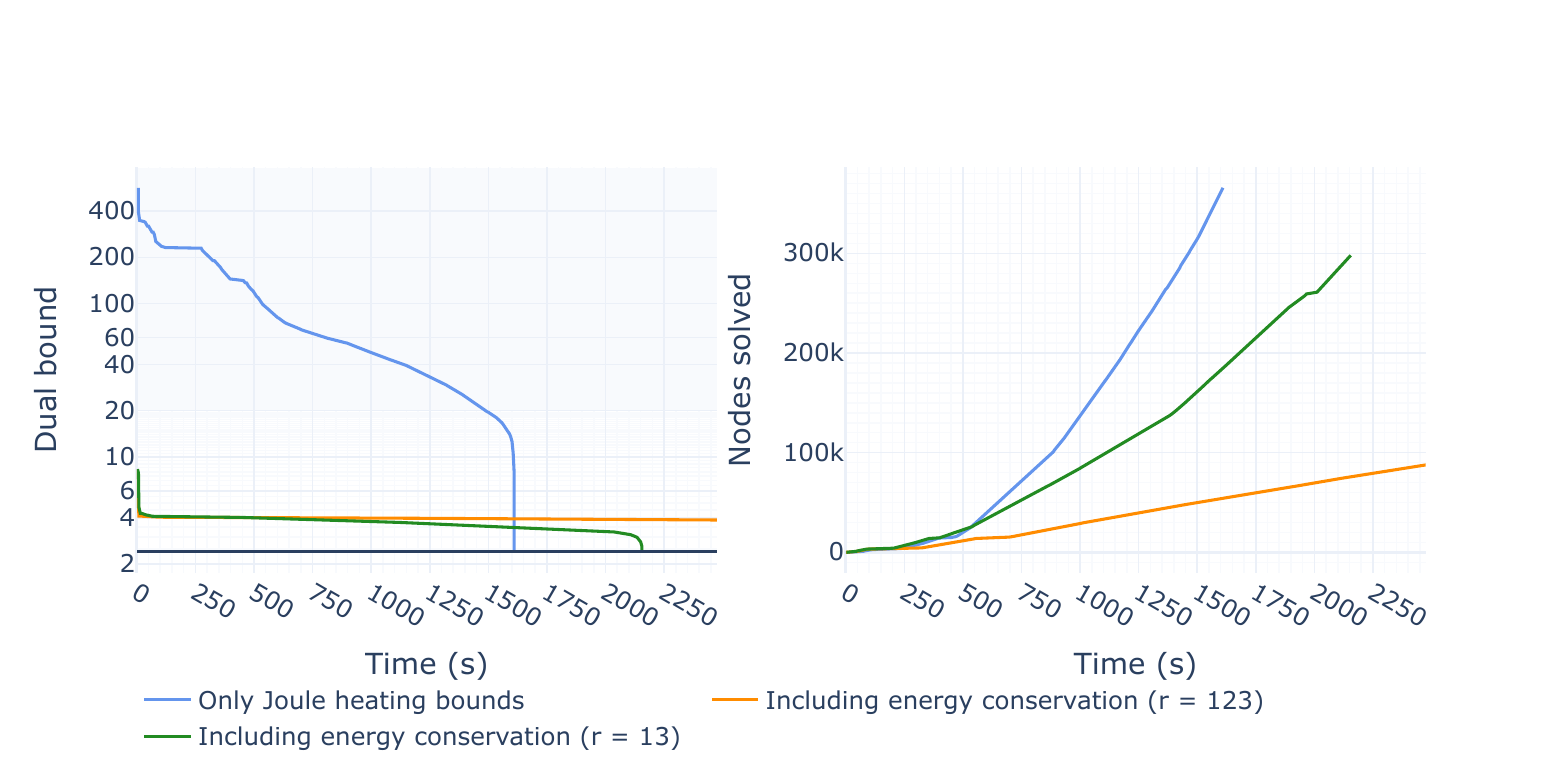}
    \caption{
        Dual bound and number of solved branch-and-bound tree nodes for the patch antenna instance of \Cref{section: patch} over time.
        The horizontal black line marks the primal optimal value.
        Note that the figure to the left has a logarithmic scale on its vertical axis.
    }
    \label{fig: dual bound}
\end{figure}

\section{Discussion}
\label{section: discussion}

One motivation for the present work is to exhibit a new set of challenging optimization problems for the mathematical programming community.
Together with the predecessor paper~\cite{hoeksma2026solving}, we have placed a difficult and practically important engineering design problem within the realm of applicability of standard mathematical programming tools.

In particular, our application of physics-informed cuts based on the conservation of energy has proved essential to drastically strengthen the basic Program \eqref{eq: base mip}.
In future applications, we expect that deriving more physics-informed constraints will help to improve the scalability of the program.

Our test case in \Cref{section: computation} shows the effectiveness of this approach.
However, from the engineering perspective, this test case is still a rather simple patch antenna design problem.
Therefore, in a third companion paper~\cite{vanrhijn2026optimal}, we evaluate the proposed methods on more realistic test cases, including the circular polarization constraints of \Cref{sec: polarization} and an extension of \Cref{thm: base mip} to multi-frequency design. That paper
targets the antenna engineering  rather than the mathematical programming community.

\begin{acknowledgements}
We are most grateful to Dave Bekers, Raymond van Dijk, Stefania Monni and Yannis Iliopoulos at the Netherlands Organisation for Applied Scientific Research, who provided us with much help and feedback over the course of this project. 

The authors received funding from the Dutch National Growth Fund programme POLARIS (\url{www.polaris-ngf.nl}).
\end{acknowledgements}

\section*{Conflict of interest and AI}

The authors declare that they have no conflict of interest.  
This paper and the results therein were obtained without substantial usage of AI tools. 
AI was exclusively used for some routine calculations, to help with literature search, and to help with programming.

\printbibliography

\end{document}

%% file: square_91.tikz
\path[stlfacet] (-2.248443,-2.248443,0.424706) -- (-1.349066,-2.248443,0.424706) -- (-1.78417,-1.780315,0.424706) -- cycle;
\path[stlfacet] (-2.248443,-1.57391,0.424706) -- (-2.248443,-2.248443,0.424706) -- (-1.78417,-1.780315,0.424706) -- cycle;
\path[stlfacet] (-1.814703,-1.271665,0.424706) -- (-2.248443,-1.57391,0.424706) -- (-1.78417,-1.780315,0.424706) -- cycle;
\path[stlfacet] (-1.349066,-2.248443,0.424706) -- (-1.231027,-1.297387,0.424706) -- (-1.78417,-1.780315,0.424706) -- cycle;
\path[stlfacet] (-2.248443,-0.899377,0.424706) -- (-2.248443,-1.57391,0.424706) -- (-1.814703,-1.271665,0.424706) -- cycle;
\path[stlfacet] (-1.349066,-2.248443,0.424706) -- (-0.449689,-2.248443,0.424706) -- (-0.472922,-1.568761,0.424706) -- cycle;
\path[stlfacet] (-1.231027,-1.297387,0.424706) -- (-1.814703,-1.271665,0.424706) -- (-1.78417,-1.780315,0.424706) -- cycle;
\path[stlfacet] (-1.231027,-1.297387,0.424706) -- (-1.349066,-2.248443,0.424706) -- (-0.472922,-1.568761,0.424706) -- cycle;
\path[stlfacet] (-3.372665,0,0.212353) -- (-3.372665,-0.103255,0.309871) -- (-3.372665,-0.224844,0.267315) -- (-3.372665,-0.224844,0) -- cycle;
\path[stlfacet] (-3.372665,0.224844,0) -- (-3.372665,0,0.212353) -- (-3.372665,-0.224844,0) -- cycle;
\path[stlfacet] (-0.449689,-2.248443,0.424706) -- (0.449689,-2.248443,0.424706) -- (-0.472922,-1.568761,0.424706) -- cycle;
\path[stlfacet] (-1.546849,-0.676474,0.424706) -- (-2.248443,-0.899377,0.424706) -- (-1.814703,-1.271665,0.424706) -- cycle;
\path[stlfacet] (-3.372665,0.224844,0.424706) -- (-3.372665,-0.103255,0.309871) -- (-3.372665,0,0.212353) -- cycle;
\path[stlfacet] (-1.231027,-1.297387,0.424706) -- (-1.546849,-0.676474,0.424706) -- (-1.814703,-1.271665,0.424706) -- cycle;
\path[stlfacet] (-2.810554,-0.224844,0.424706) -- (-2.810554,0,0.424706) -- (-3.372665,-0.224844,0.424706) -- cycle;
\path[stlfacet] (-3.372665,0.224844,0.424706) -- (-3.372665,0,0.212353) -- (-3.372665,0.224844,0) -- cycle;
\path[stlfacet] (-3.372665,0.224844,0.424706) -- (-3.372665,-0.224844,0.424706) -- (-2.810554,0,0.424706) -- cycle;
\path[stlfacet] (-2.248443,-0.224844,0.424706) -- (-2.248443,-0.899377,0.424706) -- (-1.546849,-0.676474,0.424706) -- cycle;
\path[stlfacet] (-2.810554,0,0.424706) -- (-2.810554,-0.224844,0.424706) -- (-2.248443,-0.224844,0.424706) -- cycle;
\path[stlfacet] (-0.79667,-0.735233,0.424706) -- (-1.231027,-1.297387,0.424706) -- (-0.472922,-1.568761,0.424706) -- cycle;
\path[stlfacet] (-2.810554,0.224844,0.424706) -- (-3.372665,0.224844,0.424706) -- (-2.810554,0,0.424706) -- cycle;
\path[stlfacet] (0.449689,-2.248443,0.424706) -- (0.657719,-1.48183,0.424706) -- (-0.472922,-1.568761,0.424706) -- cycle;
\path[stlfacet] (-1.546849,-0.676474,0.424706) -- (-1.231027,-1.297387,0.424706) -- (-0.79667,-0.735233,0.424706) -- cycle;
\path[stlfacet] (0.449689,-2.248443,0.424706) -- (1.349066,-2.248443,0.424706) -- (0.657719,-1.48183,0.424706) -- cycle;
\path[stlfacet] (-2.248443,0.224844,0.424706) -- (-2.810554,0,0.424706) -- (-2.248443,-0.224844,0.424706) -- cycle;
\path[stlfacet] (-2.810554,0,0.424706) -- (-2.248443,0.224844,0.424706) -- (-2.810554,0.224844,0.424706) -- cycle;
\path[stlfacet] (-0.042221,-0.894517,0.424706) -- (-0.79667,-0.735233,0.424706) -- (-0.472922,-1.568761,0.424706) -- cycle;
\path[stlfacet] (0.657719,-1.48183,0.424706) -- (-0.042221,-0.894517,0.424706) -- (-0.472922,-1.568761,0.424706) -- cycle;
\path[stlfacet] (-1.124222,0,0.424706) -- (-2.248443,-0.224844,0.424706) -- (-1.546849,-0.676474,0.424706) -- cycle;
\path[stlfacet] (-1.124222,0,0.424706) -- (-1.546849,-0.676474,0.424706) -- (-0.79667,-0.735233,0.424706) -- cycle;
\path[stlfacet] (-2.248443,0.224844,0.424706) -- (-2.248443,-0.224844,0.424706) -- (-1.124222,0,0.424706) -- cycle;
\path[stlfacet] (0.657719,-1.48183,0.424706) -- (1.349066,-2.248443,0.424706) -- (1.586749,-1.62102,0.424706) -- cycle;
\path[stlfacet] (1.349066,-2.248443,0.424706) -- (2.248443,-2.248443,0.424706) -- (1.586749,-1.62102,0.424706) -- cycle;
\path[stlfacet] (0.629062,-0.680291,0.424706) -- (-0.042221,-0.894517,0.424706) -- (0.657719,-1.48183,0.424706) -- cycle;
\path[stlfacet] (-0.042221,-0.894517,0.424706) -- (-0.281055,0,0.424706) -- (-0.79667,-0.735233,0.424706) -- cycle;
\path[stlfacet] (1.174187,-1.124222,0.424706) -- (0.657719,-1.48183,0.424706) -- (1.586749,-1.62102,0.424706) -- cycle;
\path[stlfacet] (-0.281055,0,0.424706) -- (-1.124222,0,0.424706) -- (-0.79667,-0.735233,0.424706) -- cycle;
\path[stlfacet] (-1.124222,0,0.424706) -- (-1.546849,0.676474,0.424706) -- (-2.248443,0.224844,0.424706) -- cycle;
\path[stlfacet] (-2.248443,0.224844,0.424706) -- (-1.546849,0.676474,0.424706) -- (-2.248443,0.899377,0.424706) -- cycle;
\path[stlfacet] (1.174187,-1.124222,0.424706) -- (0.629062,-0.680291,0.424706) -- (0.657719,-1.48183,0.424706) -- cycle;
\path[stlfacet] (2.248443,-2.248443,0.424706) -- (2.248443,-1.498962,0.424706) -- (1.586749,-1.62102,0.424706) -- cycle;
\path[stlfacet] (-0.281055,0,0.424706) -- (-0.042221,-0.894517,0.424706) -- (0.629062,-0.680291,0.424706) -- cycle;
\path[stlfacet] (1.174187,-1.124222,0.424706) -- (1.586749,-1.62102,0.424706) -- (1.806688,-1.124129,0.424706) -- cycle;
\path[stlfacet] (1.586749,-1.62102,0.424706) -- (2.248443,-1.498962,0.424706) -- (1.806688,-1.124129,0.424706) -- cycle;
\path[stlfacet] (-1.124222,0,0.424706) -- (-0.79667,0.735233,0.424706) -- (-1.546849,0.676474,0.424706) -- cycle;
\path[stlfacet] (-0.281055,0,0.424706) -- (-0.79667,0.735233,0.424706) -- (-1.124222,0,0.424706) -- cycle;
\path[stlfacet] (-1.546849,0.676474,0.424706) -- (-1.814703,1.271665,0.424706) -- (-2.248443,0.899377,0.424706) -- cycle;
\path[stlfacet] (0.629062,-0.680291,0.424706) -- (1.174187,-1.124222,0.424706) -- (1.603059,-0.626517,0.424706) -- cycle;
\path[stlfacet] (0.562111,0,0.424706) -- (-0.281055,0,0.424706) -- (0.629062,-0.680291,0.424706) -- cycle;
\path[stlfacet] (1.603059,-0.626517,0.424706) -- (1.174187,-1.124222,0.424706) -- (1.806688,-1.124129,0.424706) -- cycle;
\path[stlfacet] (-2.248443,0.899377,0.424706) -- (-1.814703,1.271665,0.424706) -- (-2.248443,1.57391,0.424706) -- cycle;
\path[stlfacet] (2.248443,-1.498962,0.424706) -- (2.248443,-0.749481,0.424706) -- (1.806688,-1.124129,0.424706) -- cycle;
\path[stlfacet] (-1.231027,1.297387,0.424706) -- (-1.814703,1.271665,0.424706) -- (-1.546849,0.676474,0.424706) -- cycle;
\path[stlfacet] (-1.546849,0.676474,0.424706) -- (-0.79667,0.735233,0.424706) -- (-1.231027,1.297387,0.424706) -- cycle;
\path[stlfacet] (1.405277,0,0.424706) -- (0.629062,-0.680291,0.424706) -- (1.603059,-0.626517,0.424706) -- cycle;
\path[stlfacet] (1.405277,0,0.424706) -- (0.562111,0,0.424706) -- (0.629062,-0.680291,0.424706) -- cycle;
\path[stlfacet] (2.248443,-0.749481,0.424706) -- (1.603059,-0.626517,0.424706) -- (1.806688,-1.124129,0.424706) -- cycle;
\path[stlfacet] (-0.042221,0.894517,0.424706) -- (-0.79667,0.735233,0.424706) -- (-0.281055,0,0.424706) -- cycle;
\path[stlfacet] (-1.814703,1.271665,0.424706) -- (-1.78417,1.780315,0.424706) -- (-2.248443,1.57391,0.424706) -- cycle;
\path[stlfacet] (0.562111,0,0.424706) -- (0.629062,0.680291,0.424706) -- (-0.281055,0,0.424706) -- cycle;
\path[stlfacet] (-1.231027,1.297387,0.424706) -- (-1.78417,1.780315,0.424706) -- (-1.814703,1.271665,0.424706) -- cycle;
\path[stlfacet] (-0.281055,0,0.424706) -- (0.629062,0.680291,0.424706) -- (-0.042221,0.894517,0.424706) -- cycle;
\path[stlfacet] (-2.248443,1.57391,0.424706) -- (-1.78417,1.780315,0.424706) -- (-2.248443,2.248443,0.424706) -- cycle;
\path[stlfacet] (-0.79667,0.735233,0.424706) -- (-0.472922,1.568761,0.424706) -- (-1.231027,1.297387,0.424706) -- cycle;
\path[stlfacet] (1.405277,0,0.424706) -- (0.629062,0.680291,0.424706) -- (0.562111,0,0.424706) -- cycle;
\path[stlfacet] (2.248443,-0.749481,0.424706) -- (2.248443,0,0.424706) -- (1.603059,-0.626517,0.424706) -- cycle;
\path[stlfacet] (-0.042221,0.894517,0.424706) -- (-0.472922,1.568761,0.424706) -- (-0.79667,0.735233,0.424706) -- cycle;
\path[stlfacet] (2.248443,0,0.424706) -- (1.405277,0,0.424706) -- (1.603059,-0.626517,0.424706) -- cycle;
\path[stlfacet] (-1.349066,2.248443,0.424706) -- (-1.78417,1.780315,0.424706) -- (-1.231027,1.297387,0.424706) -- cycle;
\path[stlfacet] (-2.248443,2.248443,0.424706) -- (-1.78417,1.780315,0.424706) -- (-1.349066,2.248443,0.424706) -- cycle;
\path[stlfacet] (-1.231027,1.297387,0.424706) -- (-0.472922,1.568761,0.424706) -- (-1.349066,2.248443,0.424706) -- cycle;
\path[stlfacet] (1.405277,0,0.424706) -- (1.603059,0.626517,0.424706) -- (0.629062,0.680291,0.424706) -- cycle;
\path[stlfacet] (2.248443,0,0.424706) -- (1.603059,0.626517,0.424706) -- (1.405277,0,0.424706) -- cycle;
\path[stlfacet] (0.629062,0.680291,0.424706) -- (0.657719,1.48183,0.424706) -- (-0.042221,0.894517,0.424706) -- cycle;
\path[stlfacet] (0.657719,1.48183,0.424706) -- (-0.472922,1.568761,0.424706) -- (-0.042221,0.894517,0.424706) -- cycle;
\path[stlfacet] (0.629062,0.680291,0.424706) -- (1.603059,0.626517,0.424706) -- (1.174187,1.124222,0.424706) -- cycle;
\path[stlfacet] (-1.349066,2.248443,0.424706) -- (-0.472922,1.568761,0.424706) -- (-0.449689,2.248443,0.424706) -- cycle;
\path[stlfacet] (1.174187,1.124222,0.424706) -- (0.657719,1.48183,0.424706) -- (0.629062,0.680291,0.424706) -- cycle;
\path[stlfacet] (2.248443,0.749481,0.424706) -- (1.603059,0.626517,0.424706) -- (2.248443,0,0.424706) -- cycle;
\path[stlfacet] (1.603059,0.626517,0.424706) -- (1.806688,1.124129,0.424706) -- (1.174187,1.124222,0.424706) -- cycle;
\path[stlfacet] (0.449689,2.248443,0.424706) -- (-0.472922,1.568761,0.424706) -- (0.657719,1.48183,0.424706) -- cycle;
\path[stlfacet] (-0.449689,2.248443,0.424706) -- (-0.472922,1.568761,0.424706) -- (0.449689,2.248443,0.424706) -- cycle;
\path[stlfacet] (2.248443,0.749481,0.424706) -- (1.806688,1.124129,0.424706) -- (1.603059,0.626517,0.424706) -- cycle;
\path[stlfacet] (1.174187,1.124222,0.424706) -- (1.586749,1.62102,0.424706) -- (0.657719,1.48183,0.424706) -- cycle;
\path[stlfacet] (1.174187,1.124222,0.424706) -- (1.806688,1.124129,0.424706) -- (1.586749,1.62102,0.424706) -- cycle;
\path[stlfacet] (2.248443,1.498962,0.424706) -- (1.806688,1.124129,0.424706) -- (2.248443,0.749481,0.424706) -- cycle;
\path[stlfacet] (0.449689,2.248443,0.424706) -- (0.657719,1.48183,0.424706) -- (1.349066,2.248443,0.424706) -- cycle;
\path[stlfacet] (0.657719,1.48183,0.424706) -- (1.586749,1.62102,0.424706) -- (1.349066,2.248443,0.424706) -- cycle;
\path[stlfacet] (1.586749,1.62102,0.424706) -- (1.806688,1.124129,0.424706) -- (2.248443,1.498962,0.424706) -- cycle;
\path[stlfacet] (2.248443,2.248443,0.424706) -- (1.586749,1.62102,0.424706) -- (2.248443,1.498962,0.424706) -- cycle;
\path[stlfacet] (1.349066,2.248443,0.424706) -- (1.586749,1.62102,0.424706) -- (2.248443,2.248443,0.424706) -- cycle;

%% file: patch_opt_91.tikz
\path[stlfacet] (-2.2484,-2.2484,0.4247) -- (-1.3491,-2.2484,0.4247) -- (-1.7842,-1.7803,0.4247) -- cycle;
\path[stlfacet] (-2.2484,-1.5739,0.4247) -- (-2.2484,-2.2484,0.4247) -- (-1.7842,-1.7803,0.4247) -- cycle;
\path[stlfacet] (-1.8147,-1.2717,0.4247) -- (-2.2484,-1.5739,0.4247) -- (-1.7842,-1.7803,0.4247) -- cycle;
\path[stlfacet] (-1.3491,-2.2484,0.4247) -- (-1.231,-1.2974,0.4247) -- (-1.7842,-1.7803,0.4247) -- cycle;
\path[stlfacet] (-1.3491,-2.2484,0.4247) -- (-0.4497,-2.2484,0.4247) -- (-0.4729,-1.5688,0.4247) -- cycle;
\path[stlfacet] (-1.231,-1.2974,0.4247) -- (-1.8147,-1.2717,0.4247) -- (-1.7842,-1.7803,0.4247) -- cycle;
\path[stlfacet] (-1.231,-1.2974,0.4247) -- (-1.3491,-2.2484,0.4247) -- (-0.4729,-1.5688,0.4247) -- cycle;
\path[stlfacet] (-3.3727,0,0.2124) -- (-3.3727,-0.10323,0.30989) -- (-3.3727,-0.2248,0.26734) -- (-3.3727,-0.2248,0) -- cycle;
\path[stlfacet] (-3.3727,0.2248,0) -- (-3.3727,0,0.2124) -- (-3.3727,-0.2248,0) -- cycle;
\path[stlfacet] (-0.4497,-2.2484,0.4247) -- (0.4497,-2.2484,0.4247) -- (-0.4729,-1.5688,0.4247) -- cycle;
\path[stlfacet] (-3.3727,0.2248,0.4247) -- (-3.3727,-0.10323,0.30989) -- (-3.3727,0,0.2124) -- cycle;
\path[stlfacet] (-2.8106,-0.2248,0.4247) -- (-2.8106,0,0.4247) -- (-3.3727,-0.2248,0.4247) -- cycle;
\path[stlfacet] (-3.3727,0.2248,0.4247) -- (-3.3727,0,0.2124) -- (-3.3727,0.2248,0) -- cycle;
\path[stlfacet] (-3.3727,0.2248,0.4247) -- (-3.3727,-0.2248,0.4247) -- (-2.8106,0,0.4247) -- cycle;
\path[stlfacet] (-2.8106,0,0.4247) -- (-2.8106,-0.2248,0.4247) -- (-2.2484,-0.2248,0.4247) -- cycle;
\path[stlfacet] (-0.7967,-0.7352,0.4247) -- (-1.231,-1.2974,0.4247) -- (-0.4729,-1.5688,0.4247) -- cycle;
\path[stlfacet] (-2.8106,0.2248,0.4247) -- (-3.3727,0.2248,0.4247) -- (-2.8106,0,0.4247) -- cycle;
\path[stlfacet] (0.4497,-2.2484,0.4247) -- (0.6577,-1.4818,0.4247) -- (-0.4729,-1.5688,0.4247) -- cycle;
\path[stlfacet] (-1.5468,-0.6765,0.4247) -- (-1.231,-1.2974,0.4247) -- (-0.7967,-0.7352,0.4247) -- cycle;
\path[stlfacet] (0.4497,-2.2484,0.4247) -- (1.3491,-2.2484,0.4247) -- (0.6577,-1.4818,0.4247) -- cycle;
\path[stlfacet] (-2.2484,0.2248,0.4247) -- (-2.8106,0,0.4247) -- (-2.2484,-0.2248,0.4247) -- cycle;
\path[stlfacet] (-2.8106,0,0.4247) -- (-2.2484,0.2248,0.4247) -- (-2.8106,0.2248,0.4247) -- cycle;
\path[stlfacet] (-0.0422,-0.8945,0.4247) -- (-0.7967,-0.7352,0.4247) -- (-0.4729,-1.5688,0.4247) -- cycle;
\path[stlfacet] (0.6577,-1.4818,0.4247) -- (-0.0422,-0.8945,0.4247) -- (-0.4729,-1.5688,0.4247) -- cycle;
\path[stlfacet] (-1.1242,0,0.4247) -- (-2.2484,-0.2248,0.4247) -- (-1.5468,-0.6765,0.4247) -- cycle;
\path[stlfacet] (-1.1242,0,0.4247) -- (-1.5468,-0.6765,0.4247) -- (-0.7967,-0.7352,0.4247) -- cycle;
\path[stlfacet] (-2.2484,0.2248,0.4247) -- (-2.2484,-0.2248,0.4247) -- (-1.1242,0,0.4247) -- cycle;
\path[stlfacet] (0.6577,-1.4818,0.4247) -- (1.3491,-2.2484,0.4247) -- (1.5867,-1.621,0.4247) -- cycle;
\path[stlfacet] (1.3491,-2.2484,0.4247) -- (2.2484,-2.2484,0.4247) -- (1.5867,-1.621,0.4247) -- cycle;
\path[stlfacet] (0.6291,-0.6803,0.4247) -- (-0.0422,-0.8945,0.4247) -- (0.6577,-1.4818,0.4247) -- cycle;
\path[stlfacet] (-1.1242,0,0.4247) -- (-1.5468,0.6765,0.4247) -- (-2.2484,0.2248,0.4247) -- cycle;
\path[stlfacet] (1.1742,-1.1242,0.4247) -- (0.6291,-0.6803,0.4247) -- (0.6577,-1.4818,0.4247) -- cycle;
\path[stlfacet] (2.2484,-2.2484,0.4247) -- (2.2484,-1.499,0.4247) -- (1.5867,-1.621,0.4247) -- cycle;
\path[stlfacet] (1.1742,-1.1242,0.4247) -- (1.5867,-1.621,0.4247) -- (1.8067,-1.1241,0.4247) -- cycle;
\path[stlfacet] (1.5867,-1.621,0.4247) -- (2.2484,-1.499,0.4247) -- (1.8067,-1.1241,0.4247) -- cycle;
\path[stlfacet] (-1.1242,0,0.4247) -- (-0.7967,0.7352,0.4247) -- (-1.5468,0.6765,0.4247) -- cycle;
\path[stlfacet] (0.6291,-0.6803,0.4247) -- (1.1742,-1.1242,0.4247) -- (1.6031,-0.6265,0.4247) -- cycle;
\path[stlfacet] (1.6031,-0.6265,0.4247) -- (1.1742,-1.1242,0.4247) -- (1.8067,-1.1241,0.4247) -- cycle;
\path[stlfacet] (-1.5468,0.6765,0.4247) -- (-0.7967,0.7352,0.4247) -- (-1.231,1.2974,0.4247) -- cycle;
\path[stlfacet] (1.4053,0,0.4247) -- (0.6291,-0.6803,0.4247) -- (1.6031,-0.6265,0.4247) -- cycle;
\path[stlfacet] (-1.8147,1.2717,0.4247) -- (-1.7842,1.7803,0.4247) -- (-2.2484,1.5739,0.4247) -- cycle;
\path[stlfacet] (-1.231,1.2974,0.4247) -- (-1.7842,1.7803,0.4247) -- (-1.8147,1.2717,0.4247) -- cycle;
\path[stlfacet] (-2.2484,1.5739,0.4247) -- (-1.7842,1.7803,0.4247) -- (-2.2484,2.2484,0.4247) -- cycle;
\path[stlfacet] (-0.7967,0.7352,0.4247) -- (-0.4729,1.5688,0.4247) -- (-1.231,1.2974,0.4247) -- cycle;
\path[stlfacet] (-0.0422,0.8945,0.4247) -- (-0.4729,1.5688,0.4247) -- (-0.7967,0.7352,0.4247) -- cycle;
\path[stlfacet] (2.2484,0,0.4247) -- (1.4053,0,0.4247) -- (1.6031,-0.6265,0.4247) -- cycle;
\path[stlfacet] (-1.3491,2.2484,0.4247) -- (-1.7842,1.7803,0.4247) -- (-1.231,1.2974,0.4247) -- cycle;
\path[stlfacet] (-2.2484,2.2484,0.4247) -- (-1.7842,1.7803,0.4247) -- (-1.3491,2.2484,0.4247) -- cycle;
\path[stlfacet] (-1.231,1.2974,0.4247) -- (-0.4729,1.5688,0.4247) -- (-1.3491,2.2484,0.4247) -- cycle;
\path[stlfacet] (1.4053,0,0.4247) -- (1.6031,0.6265,0.4247) -- (0.6291,0.6803,0.4247) -- cycle;
\path[stlfacet] (2.2484,0,0.4247) -- (1.6031,0.6265,0.4247) -- (1.4053,0,0.4247) -- cycle;
\path[stlfacet] (0.6291,0.6803,0.4247) -- (0.6577,1.4818,0.4247) -- (-0.0422,0.8945,0.4247) -- cycle;
\path[stlfacet] (0.6577,1.4818,0.4247) -- (-0.4729,1.5688,0.4247) -- (-0.0422,0.8945,0.4247) -- cycle;
\path[stlfacet] (0.6291,0.6803,0.4247) -- (1.6031,0.6265,0.4247) -- (1.1742,1.1242,0.4247) -- cycle;
\path[stlfacet] (-1.3491,2.2484,0.4247) -- (-0.4729,1.5688,0.4247) -- (-0.4497,2.2484,0.4247) -- cycle;
\path[stlfacet] (1.1742,1.1242,0.4247) -- (0.6577,1.4818,0.4247) -- (0.6291,0.6803,0.4247) -- cycle;
\path[stlfacet] (1.6031,0.6265,0.4247) -- (1.8067,1.1241,0.4247) -- (1.1742,1.1242,0.4247) -- cycle;
\path[stlfacet] (0.4497,2.2484,0.4247) -- (-0.4729,1.5688,0.4247) -- (0.6577,1.4818,0.4247) -- cycle;
\path[stlfacet] (-0.4497,2.2484,0.4247) -- (-0.4729,1.5688,0.4247) -- (0.4497,2.2484,0.4247) -- cycle;
\path[stlfacet] (1.1742,1.1242,0.4247) -- (1.8067,1.1241,0.4247) -- (1.5867,1.621,0.4247) -- cycle;
\path[stlfacet] (0.4497,2.2484,0.4247) -- (0.6577,1.4818,0.4247) -- (1.3491,2.2484,0.4247) -- cycle;
\path[stlfacet] (0.6577,1.4818,0.4247) -- (1.5867,1.621,0.4247) -- (1.3491,2.2484,0.4247) -- cycle;
\path[stlfacet] (1.5867,1.621,0.4247) -- (1.8067,1.1241,0.4247) -- (2.2484,1.499,0.4247) -- cycle;
\path[stlfacet] (2.2484,2.2484,0.4247) -- (1.5867,1.621,0.4247) -- (2.2484,1.499,0.4247) -- cycle;
\path[stlfacet] (1.3491,2.2484,0.4247) -- (1.5867,1.621,0.4247) -- (2.2484,2.2484,0.4247) -- cycle;